\pdfoutput=1
\documentclass[11pt, a4paper, twoside,leqno]{amsart}
\usepackage[centering, totalwidth = 380pt, totalheight = 600pt]{geometry}
\usepackage{amssymb, amsmath, amsthm, listings, xspace}
\usepackage{enumerate, microtype, stmaryrd, url} 
\usepackage[latin1]{inputenc}
\usepackage[arrow, matrix, tips, curve, graph, rotate]{xy}
 \usepackage{color}
\definecolor{darkgreen}{rgb}{0,0.45,0}
\usepackage[colorlinks,citecolor=darkgreen,linkcolor=darkgreen]{hyperref}
\SelectTips{cm}{10}

 \DeclareMathOperator{\ob}{ob}

\DeclareMathOperator{\co}{c} 

\newcommand{\cat}[1]{\mathbf{#1}}
\newcommand{\op}{\mathrm{op}}
\renewcommand{\co}{\mathrm{co}}

\newcommand{\thg}{{\mathord{\text{--}}}}

\renewcommand{\c}{,\,}

\DeclareMathAlphabet      {\mathbf}{OT1}{cmr}{b}{n}

\newcommand{\Wkl}{\cat{Wk}_\ell(\mathsf L, \mathsf R)}
\newcommand{\Wkr}{\cat{Wk}_r(\mathsf L, \mathsf R)}

\newcommand{\abs}[1]{{\left|{#1}\right|}}

\newcommand{\spn}[1]{{\langle{#1}\rangle}}

\newcommand{\dcat}[1]{\cat{\mathbb #1}}
\newcommand{\cd}[2][]{\vcenter{\hbox{\xymatrix#1{#2}}}}
\newcommand{\fib}{{\mathsf F}}
\newcommand{\cof}{{\mathsf Q}}
\newcommand{\TAlgs}{\mathsf{T}\text-\cat{Alg}_{s}}
\newcommand{\TAlgl}{\mathsf{T}\text-\cat{Alg}_{l}}
\newcommand{\TAlgp}{\mathsf{T}\text-\cat{Alg}_{p}}

\newcommand{\TAlgw}{\mathsf{T}\text-\cat{Alg}_{w}}

\makeatletter
\def\matrixobject@{%
   \edef \next@{={\DirectionfromtheDirection@ }}%
   \expandafter \toks@ \next@ \plainxy@
   \let\xy@@ix@=\xyq@@toksix@
   \xyFN@ \OBJECT@}
\let\xy@entry@@norm=\entry@@norm
\def\entry@@norm@patched{%
   \let\object@=\matrixobject@
   \xy@entry@@norm }
\AtBeginDocument{\let\entry@@norm\entry@@norm@patched}
\makeatother

\renewcommand{\b}[1]{\boldsymbol{#1}}
\renewcommand{\phi}{\varphi}
\newcommand{\A}{{\mathcal A}}
\newcommand{\B}{{\mathcal B}}
\newcommand{\C}{{\mathcal C}}
\newcommand{\D}{{\mathcal D}}

\renewcommand{\L}{{\mathcal L}}
\newcommand{\M}{{\mathcal M}}

\renewcommand{\O}{{\mathcal O}}

\newcommand{\Ss}{{\mathcal S}}

\newcommand{\W}{{\mathcal W}}

\newcommand{\xtor}[1]{\cdl[@1]{{} \ar[r]|-{\object@{|}}^{#1} & {}}}

\makeatletter

\newcommand{\setmanuallabel}[1]{\stepcounter{equation}{\edef\@currentlabel{\theequation}\label{#1}}}
\newcommand{\printmanuallabel}[1]{\stepcounter{equation}\text{(\theequation)}}

\def\hookleftarrowfill@{\arrowfill@\leftarrow\relbar{\relbar\joinrel\rhook}}
\def\twoheadleftarrowfill@{\arrowfill@\twoheadleftarrow\relbar\relbar}
\def\leftbararrowfill@{\arrowdoublefill@{\leftarrow\mkern-5mu}\relbar\mapstochar\relbar\relbar}
\def\Leftbararrowfill@{\arrowdoublefill@{\Leftarrow\mkern-2mu}\Relbar\Mapstochar\Relbar\Relbar}
\def\leftringarrowfill@{\arrowdoublefill@{\leftarrow\mkern-3mu}\relbar{\mkern-3mu\circ\mkern-2mu}\relbar\relbar}
\def\lefttriarrowfill@{\arrowfill@{\mathrel\triangleleft\mkern0.5mu\joinrel\relbar}\relbar\relbar}
\def\Lefttriarrowfill@{\arrowfill@{\mathrel\triangleleft\mkern1mu\joinrel\Relbar}\Relbar\Relbar}

\def\hookrightarrowfill@{\arrowfill@{\lhook\joinrel\relbar}\relbar\rightarrow}
\def\twoheadrightarrowfill@{\arrowfill@\relbar\relbar\twoheadrightarrow}
\def\rightbararrowfill@{\arrowdoublefill@{\relbar\mkern-0.5mu}\relbar\mapstochar\relbar\rightarrow}
\def\Rightbararrowfill@{\arrowdoublefill@{\Relbar\mkern-2mu}\Relbar\Mapstochar\Relbar\Rightarrow}
\def\rightringarrowfill@{\arrowdoublefill@\relbar\relbar{\mkern-2mu\circ\mkern-3mu}\relbar{\mkern-3mu\rightarrow}}
\def\righttriarrowfill@{\arrowfill@\relbar\relbar{\relbar\joinrel\mkern0.5mu\mathrel\triangleright}}
\def\Righttriarrowfill@{\arrowfill@\Relbar\Relbar{\Relbar\joinrel\mkern1mu\mathrel\triangleright}}

\def\leftrightarrowfill@{\arrowfill@\leftarrow\relbar\rightarrow}
\def\mapstofill@{\arrowfill@{\mapstochar\relbar}\relbar\rightarrow}

\renewcommand*\xleftarrow[2][]{\ext@arrow 20{20}0\leftarrowfill@{#1}{#2}}
\providecommand*\xLeftarrow[2][]{\ext@arrow 60{22}0{\Leftarrowfill@}{#1}{#2}}
\providecommand*\xhookleftarrow[2][]{\ext@arrow 10{20}0\hookleftarrowfill@{#1}{#2}}
\providecommand*\xtwoheadleftarrow[2][]{\ext@arrow 60{20}0\twoheadleftarrowfill@{#1}{#2}}
\providecommand*\xleftbararrow[2][]{\ext@arrow 10{22}0\leftbararrowfill@{#1}{#2}}
\providecommand*\xLeftbararrow[2][]{\ext@arrow 50{24}0\Leftbararrowfill@{#1}{#2}}
\providecommand*\xleftringarrow[2][]{\ext@arrow 10{26}0\leftringarrowfill@{#1}{#2}}
\providecommand*\xlefttriarrow[2][]{\ext@arrow 80{24}0\lefttriarrowfill@{#1}{#2}}
\providecommand*\xLefttriarrow[2][]{\ext@arrow 80{24}0\Lefttriarrowfill@{#1}{#2}}

\renewcommand*\xrightarrow[2][]{\ext@arrow 01{20}0\rightarrowfill@{#1}{#2}}
\providecommand*\xRightarrow[2][]{\ext@arrow 04{22}0{\Rightarrowfill@}{#1}{#2}}
\providecommand*\xhookrightarrow[2][]{\ext@arrow 00{20}0\hookrightarrowfill@{#1}{#2}}
\providecommand*\xtwoheadrightarrow[2][]{\ext@arrow 03{20}0\twoheadrightarrowfill@{#1}{#2}}
\providecommand*\xrightbararrow[2][]{\ext@arrow 01{22}0\rightbararrowfill@{#1}{#2}}
\providecommand*\xRightbararrow[2][]{\ext@arrow 04{24}0\Rightbararrowfill@{#1}{#2}}
\providecommand*\xrightringarrow[2][]{\ext@arrow 01{26}0\rightringarrowfill@{#1}{#2}}
\providecommand*\xrighttriarrow[2][]{\ext@arrow 07{24}0\righttriarrowfill@{#1}{#2}}
\providecommand*\xRighttriarrow[2][]{\ext@arrow 07{24}0\Righttriarrowfill@{#1}{#2}}

\providecommand*\xmapsto[2][]{\ext@arrow 01{20}0\mapstofill@{#1}{#2}}
\providecommand*\xleftrightarrow[2][]{\ext@arrow 10{22}0\leftrightarrowfill@{#1}{#2}}
\providecommand*\xLeftrightarrow[2][]{\ext@arrow 10{27}0{\Leftrightarrowfill@}{#1}{#2}}

\makeatother

\newcommand{\twocong}[2][0.5]{\ar@{}[#2] \save ?(#1)*{\cong}\restore}
\newcommand{\twoeq}[2][0.5]{\ar@{}[#2] \save ?(#1)*{=}\restore}
\newcommand{\rtwocell}[3][0.5]{\ar@{}[#2] \ar@{=>}?(#1)+/l 0.2cm/;?(#1)+/r 0.2cm/^{#3}}
\newcommand{\ltwocell}[3][0.5]{\ar@{}[#2] \ar@{=>}?(#1)+/r 0.2cm/;?(#1)+/l 0.2cm/^{#3}}
\newcommand{\ltwocello}[3][0.5]{\ar@{}[#2] \ar@{=>}?(#1)+/r 0.2cm/;?(#1)+/l 0.2cm/_{#3}}
\newcommand{\dtwocell}[3][0.5]{\ar@{}[#2] \ar@{=>}?(#1)+/u  0.2cm/;?(#1)+/d 0.2cm/^{#3}}
\newcommand{\dltwocell}[3][0.5]{\ar@{}[#2] \ar@{=>}?(#1)+/ur  0.2cm/;?(#1)+/dl 0.2cm/^{#3}}
\newcommand{\drtwocell}[3][0.5]{\ar@{}[#2] \ar@{=>}?(#1)+/ul  0.2cm/;?(#1)+/dr 0.2cm/^{#3}}
\newcommand{\dthreecell}[3][0.5]{\ar@{}[#2] \ar@3{->}?(#1)+/u  0.2cm/;?(#1)+/d 0.2cm/^{#3}}
\newcommand{\utwocell}[3][0.5]{\ar@{}[#2] \ar@{=>}?(#1)+/d 0.2cm/;?(#1)+/u 0.2cm/_{#3}}
\newcommand{\dtwocelltarg}[3][0.5]{\ar@{}#2 \ar@{=>}?(#1)+/u  0.2cm/;?(#1)+/d 0.2cm/^{#3}}
\newcommand{\utwocelltarg}[3][0.5]{\ar@{}#2 \ar@{=>}?(#1)+/d  0.2cm/;?(#1)+/u 0.2cm/_{#3}}

\newcommand{\pushoutcorner}[1][dr]{\save*!/#1-1.2pc/#1:(-1,1)@^{|-}\restore}

\newdir{(}{{}*!<0em,-.14em>-\cir<.14em>{l^r}}
\newdir{ (}{{}*!/-5pt/\dir{(}}
\newdir{ >}{{}*!/-5pt/\dir{>}}

\theoremstyle{definition}

\theoremstyle{plain}
\newtheorem{Thm}{Theorem}
\newtheorem{Prop}[Thm]{Proposition}

\newtheorem{Cor}[Thm]{Corollary}
\newtheorem{Lemma}[Thm]{Lemma}
\numberwithin{equation}{section}

\theoremstyle{definition}
\newtheorem{Defn}[Thm]{Definition}

\newtheorem{Ex}[Thm]{Example}
\newtheorem{Exs}[Thm]{Examples}
\newtheorem{Rk}[Thm]{Remark}

\renewcommand{\l}[1]{L{#1}}
\newcommand{\fl}[1]{\alg{L}{#1}}
\renewcommand{\r}[1]{R{#1}}
\newcommand{\fr}[1]{\alg{R}{#1}}
\newcommand{\m}[1]{\mu_{#1}}
\renewcommand{\c}[1]{\Delta_{#1}}

\newcommand{\lp}[1]{L'{#1}}

\newcommand{\rp}[1]{R'{#1}}

\newcommand{\Coalg}[1]{\mathsf{#1}\text-\cat{Coalg}}
\newcommand{\Alg}[1]{\mathsf{#1}\text-\cat{Alg}}
\newcommand{\Kl}[1]{\cat{Kl}(\mathsf{#1})}

\newcommand{\DAlg}[1]{\mathsf{#1}\text-\mathbb A\cat{lg}}
\newcommand{\Sq}[1]{\mathbb S\cat{q}(#1)}
\newcommand{\aone}{{\mathbf 1}}
\newcommand{\atwo}{{\mathbf 2}}

\newcommand{\alg}[1]{\boldsymbol{#1}}
\newcommand{\awfs}{\textsc{awfs}\xspace}

\newcommand{\Awfs}[1]{{\cat{AWFS}(#1)}}
\newcommand{\Lax}{{\cat{AWFS}_\mathrm{lax}}}
\newcommand{\Oplax}{{\cat{AWFS}_\mathrm{oplax}}}
\newcommand{\Ladj}{{\cat{AWFS}_\mathrm{ladj}}}
\newcommand{\Radj}{{\cat{AWFS}_\mathrm{radj}}}

\newcommand{\Cat}{{\cat{CAT}}}

\newcommand{\Dbl}{{\cat{DBL}}}
\newcommand{\mn}{\mathbin{\text-}}

\begin{document}
\leftmargini=2em \title[Algebraic weak factorisation systems
II]{Algebraic weak factorisation systems II:\\ Categories of weak maps}
\author{John Bourke}
\address{Department of Mathematics and Statistics, Masaryk University, Kotl\'a\v rsk\'a 2, Brno 60000, Czech Republic}
\email{bourkej@math.muni.cz} 
\author{Richard Garner} \address{Department of Mathematics, Macquarie
  University, NSW 2109, Australia} \email{richard.garner@mq.edu.au}

\date{\today}

\thanks{The first author acknowledges the support of the Grant agency
  of the Czech Republic, grant number P201/12/G028. The second author
  acknowledges the support of an Australian Research Council Discovery
  Project grants DP110102360 and DP130101969.}

\maketitle

\begin{abstract}
  We investigate the categories of \emph{weak maps} associated to an
  algebraic weak factorisation system (\awfs) in the sense of
  Grandis--Tholen~\cite{Grandis2006Natural}. For any \awfs on a
  category with an initial object, cofibrant replacement forms a
  comonad, and the category of (left) weak maps associated to the
  \awfs is by definition the Kleisli category of this comonad. We
  exhibit categories of weak maps as a kind of ``homotopy category'',
  that freely adjoins a section for every ``acyclic fibration''
  (=right map) of the \awfs; and using this characterisation, we give
  an alternate description of categories of weak maps in terms of
  spans with left leg an acyclic fibration. We moreover show that the
  $2$-functor sending each \awfs on a suitable category to its
  cofibrant replacement comonad has a fully faithful right adjoint: so
  exhibiting the theory of comonads, and dually of monads, as
  incorporated into the theory of \awfs. We also describe various
  applications of the general theory: to the generalised sketches of
  Kinoshita--Power--Takeyama~\cite{Kinoshita1999Sketches}, to the
  two-dimensional monad theory of
  Blackwell--Kelly--Power~\cite{Blackwell1989Two-dimensional}, and to
  the theory of dg-categories~\cite{Kelly1965Chain}.
\end{abstract}

\section{Introduction}
This paper continues the authors' ongoing investigations into the
\emph{algebraic weak factorisation systems}
of~\cite{Grandis2006Natural,Garner2009Understanding,Riehl2011Algebraic,Athorne2012The-coalgebraic}.
An algebraic weak factorisation system (henceforth \awfs) is a weak
factorisation system in which each map $f$ is equipped with a
factorisation $f = Rf \cdot Lf$, in such a way that the assignations
$f \mapsto Lf$ and $f \mapsto Rf$ become the actions on objects of an
interacting comonad $\mathsf L$ and monad $\mathsf R$ on the arrow
category. This extra structure allows algebraic weak factorisation
systems to do things that mere weak factorisation systems cannot; for
example, the first paper in this series~\cite{Bourke2014AWFS1}
constructed an \awfs on the category of
\emph{quasicategories}~\cite{Joyal2002Quasi-categories,Joyal2008The-theory,Lurie2009Higher}
whose (algebraically) fibrant objects are quasicategories with finite
limits; and another whose (algebraic) fibrations are the Grothendieck
fibrations of quasicategories.

In this paper, we study a further aspect of the theory of \awfs which
is specifically enabled by the algebraic perspective. Any weak
factorisation system on a category $\C$ with an initial object gives
rise to a notion of ``cofibrant replacement'' by factorising the
unique maps out of the initial object, and in the algebraic case, this
cofibrant replacement underlies a comonad $\mathsf Q$. For suitable
choices of \awfs, the Kleisli category of $\mathsf Q$---wherein maps
$A \rightsquigarrow B$ are maps $QA \to B$ in the original
category---will equip $\C$ with a usable notion of \emph{weak map}.
For example:
\begin{enumerate}[(i)]
\item There is an \awfs on the category of
  tricategories~\cite{Gordon1995Coherence} and strict morphisms
  (preserving all structure on the nose) for which $\Kl{Q}$ comprises
  the tricategories and their trihomomorphisms (preserving all
  structure up to coherent equivalence);
  see~\cite{Garner2010Homomorphisms}.\vskip0.25\baselineskip
\item If $\mathsf T$ is an accessible $2$-monad on a complete and
  cocomplete $2$-category, then there are \awfs on the category
  $\TAlgs$ of $\mathsf T$-algebras and strict morphisms such that the
  corresponding
  $\Kl{Q}$ is the category $\TAlgp$ of $\mathsf T$-algebras and
  algebra pseudomorphisms, respectively the category $\TAlgl$ of
  $\mathsf T$-algebras and lax algebra morphisms; see
  Section~\ref{sec:two-dimens-monad} below.\vskip0.25\baselineskip
\item If $\mathsf T$ is a sufficiently cocontinuous dg-monad on a
  cocomplete dg-category, then there is an \awfs on $\TAlgs$, the
  category of $\mathsf T$-algebras and strict morphisms, for which
  $\Kl{Q}$ is the category $\TAlgw$ of $\mathsf T$-algebras and 
  homotopy-coherent algebra morphisms; see
  Section~\ref{sec:homological} below.
\end{enumerate}

Guided by these examples, we are led to define the \emph{category of
  (left) weak maps} $\Wkl$ of an \awfs as the Kleisli category of
its cofibrant replacement comonad. (Dually, we have categories of
\emph{right} weak maps associated to fibrant replacement monads, but
this plays only a minor role here.) While the value of the
construction is apparent from the applications listed above,
the abstract role it plays is less obvious; an important objective of
this paper is to clarify this.

Our first main result, Theorem~\ref{thm:WeakMaps},
characterises the category of weak maps as a kind of ``homotopy
category''. Recall that the \emph{homotopy category} of a Quillen
model category $\C$ is the category $\C[\W^{-1}]$ obtained by freely
inverting each weak equivalence. Similarly, the category of
weak maps of an \awfs $(\mathsf L, \mathsf R)$ arises by ``freely
splitting each $\mathsf R$-map''; thus, for each $\mathsf R$-algebra
structure $\alg f$ on a morphism $f \colon A \to B$, we adjoin a
section $\alg f^\ast$ for $f$, subject to certain coherence axioms.

Given this characterisation, we may associate a category of weak maps
to an \awfs even in the absence of an initial object---by
\emph{defining} it in terms of the universal role it fulfils. Our
second main result, Theorem~\ref{thm:pullback}, exploits this to give
(in the presence of pullbacks) a second construction of categories of
weak maps, wherein morphisms are equivalence-classes of spans with
left leg an $\mathsf R$-algebra; this is the analogue of Gabriel and
Zisman's representation~\cite{Gabriel1967Calculus} of morphisms in a
localisation $\C[\W^{-1}]$ as equivalences classes of spans with left
leg in $\W$.

Our next main result, Theorem~\ref{thm:9}, turns the universal
property of the category of weak maps into that of a
\emph{$2$-adjunction} between $2$-categories of \awfs and of comonads
on categories with finite coproducts. The left adjoint sends an \awfs
to its cofibrant replacement comonad. The right adjoint sends a
comonad $\mathsf P \colon \C \to \C$ to the \emph{$\mathsf P$-split
  epi} \awfs on $\C$, whose $\mathsf R$-algebras are maps of $\C$
equipped with a retraction in $\Kl{P}$; this is an \awfs with
cofibrant replacement comonad $\mathsf P$ and is in fact universal
among such. This $2$-adjunction exhibits the $2$-category of comonads
as a full, reflective sub-$2$-category of the $2$-category of \awfs,
so that the theory of \awfs fully incorporates that of comonads---and
dually, monads.

The results described above provide not only an abstract
characterisation of categories of weak maps, but also a useful tool
for calculating them. We illustrate this in the final two sections of
the paper by using our results to prove the claims made in (ii) and
(iii) above. We deal with (ii) in Section~\ref{sec:two-dimens-monad}:
thus, for a suitable $2$-monad $\mathsf T \colon \C \to \C$, we define
\awfs on $\TAlgs$ whose categories of weak maps are respectively
$\TAlgp$ and $\TAlgl$. These \awfs will be obtained by projectively
lifting the \awfs on $\C$ for \emph{retract equivalences} and for
\emph{left adjoint left inverse} functors. In the pseudo case, the
lifted \awfs on $\TAlgs$ is part of a model structure on $\TAlgs$,
described in~\cite[Theorem~4.5]{Lack2007Homotopy-theoretic}. 

Finally, in Section~\ref{sec:homological} we turn to the case (iii) of
dg-monads. Given a suitable dg-monad $\mathsf T \colon \C \to \C$ on a
dg-category, we define an \awfs on $\TAlgs$ whose category of weak
maps comprises $\mathsf T$-algebras and their \emph{homotopy-coherent}
morphisms. For example, when $\A$ is a small dg-category and
$\mathsf T$ is the monad on $\cat{DG}^{\ob \A}$ whose algebras are
dg-modules over $\A$, we obtain the category of dg-modules and
homotopy-coherent transformations described
in~\cite[Example~6.6]{Keller1994Deriving}. The \awfs in question is
again obtained by projectively lifting from $\C$ and is part of a
model structure on $\TAlgs$ of the kind constructed
in~\cite{Christensen2002Quillen}. Of particular note is the fact that
its cofibrant replacement comonad is precisely the classical \emph{bar
  construction}~\cite[X,~\S 6]{Cartan1956Homological}; this clarifies
the universal role that this construction fulfils. In future work we
will describe a generalisation of these results from dg-monads to
dg-operads using a \emph{dendroidal}~\cite{Moerdijk2007Dendroidal}
analogue of the bar construction, closely related to the bar--cobar
construction for operad algebras~\cite[Chapter~11]{Loday2012Algebraic}.

\section{Background material on algebraic weak factorisation systems}
In this preliminary section we summarise from~\cite[\S
2--3]{Bourke2014AWFS1} those aspects of the theory of algebraic weak
factorisation systems necessary for the present paper. We will say
enough so as to make our treatment self-contained for someone familiar
with the basic notions; for a full treatment, with proofs, we refer
the reader to~\cite{Bourke2014AWFS1} and the sources cited therein.

Before we begin, a brief note on foundational matters. We fix a
Grothendieck universe $\kappa$, and call arbitrary sets \emph{large}
and ones in $\kappa$, \emph{small}. $\cat{Set}$ and $\cat{SET}$ denote
the categories of small and large sets, while $\cat{Cat}$ and $\Cat$
are the $2$-categories of internal categories in $\cat{Set}$ and
$\cat{SET}$; in particular, objects of $\Cat$ are \emph{not} assumed
to be locally small. By a \emph{double category}, we mean an internal
category in $\Cat$, and we write $\Dbl$ for the $2$-category of
double categories and internal functors and internal natural
transformations between such.

\subsection{Algebraic weak factorisation systems and their morphisms}\label{sec:algebr-weak-fact}
An algebraic weak factorisation system on $\C$ begins with a
\emph{functorial factorisation}: a functor $\C^{\atwo} \to \C^{\mathbf
  3}$ from the category of arrows to that of composable pairs which is
a section for the composition map $\C^{\mathbf 3} \to \C^{\atwo}$. The
action of this functor at an object $f$ or morphism $(h,k) \colon f
\to g$ of $\C^{\atwo}$ is depicted as on the left or right in:
\[
f = X \xrightarrow{\l f} Ef \xrightarrow{\r f} Y
\qquad \quad \qquad
\cd{
 X
  \ar[r]^{\l f} \ar[d]_{h} & Ef \ar[d]|{E(h, k)} \ar[r]^{\r f} & Y
  \ar[d]^{k} \\
 W \ar[r]^{\l g} & Eg \ar[r]^{\r g} & Z\rlap{ .}}
\]
From these data we obtain endofunctors $L, R \colon \C^{\atwo} \to
\C^\atwo$, together with natural transformations $\epsilon \colon L
\Rightarrow 1$ and $\eta \colon 1 \Rightarrow R$ with respective $f$-components:
\begin{equation}
\cd{
A \ar[r]^1 \ar[d]_{\l f} & A \ar[d]^{f} \\
Ef \ar[r]^{\r f} & B
} \qquad \text{and} \qquad
\cd{
A \ar[r]^{\l f} \ar[d]_{f} & Ef \ar[d]^{\r f} \\
B \ar[r]^{1} & B\rlap{ .}
}\label{eq:4}
\end{equation}

An \emph{algebraic weak factorisation system} $(\mathsf L, \mathsf R)$
on $\C$ is a functorial factorisation as above, together with natural
transformations $\Delta \colon L \to LL$ and $\mu \colon RR \to R$
making $\mathsf L = (L,\epsilon, \Delta)$ and $\mathsf R =
(R,\eta,\mu)$ into a comonad and a monad respectively. The monad and
comonad axioms, together with the form~\eqref{eq:4} of $\eta$ and
$\epsilon$, force the components of $\Delta$ and $\mu$ at $f$ to be as
on the left and right in
\[
 \cd{
A \ar[d]_{\l f} \ar[r]^1 & A \ar[d]^{\l{\l f}} \\
Ef \ar[r]^{\c f} & E\l f}
\qquad 
\cd{
Ef \ar[d]_{\l{\r f}} \ar[r]^{\c f} & E\l f
\ar[d]^{\r{\l f}} \\
E\r f \ar[r]^{\m f} & Ef
}
\qquad 
\cd{
E \r f \ar[r]^{\m f} \ar[d]_{\r{\r f}} & Ef \ar[d]^{\r f} \\
B \ar[r]^1 & B\rlap{ ,}}
\]
and imply moreover that the middle square is the component at $f$ of a
natural transformation $\delta \colon LR \Rightarrow RL$. The final
axiom for an \awfs is that this $\delta$ should constitute a
\emph{distributive law} in the sense of~\cite{Beck1969Distributive} of
$\mathsf L$ over $\mathsf R$.


We now turn to morphisms between \awfs. As with other kinds of
algebraic structure borne by categories (for example monoidal
structure) various kinds of morphisms arise---strict, pseudo, lax and
oplax---and in the case of \awfs it is the lax and oplax ones that
play the central role. A \emph{lax morphism} between \awfs $(\mathsf
L, \mathsf R)$ and $(\mathsf L', \mathsf R')$ on categories $\C$ and
$\D$ is given by a functor $F \colon \C \to \D$ and a natural family
of maps $\alpha_f$ rendering commutative the left hand square in:
$$  \cd[@-1em@C-0.5em]{
    & FA \ar[dl]_{\lp {Ff}} \ar[dr]^{F\l f} \\
    E'Ff \ar[rr]^{\alpha_f} \ar[dr]_{\rp {Ff}} & & 
    FEf\rlap{ } \ar[dl]^{F\r f} \\ & FB} \qquad \qquad
  \cd[@!@-0.8em]{
    E'Ff \ar[r]^{\alpha_f} \ar[d]_{E'(\gamma_A,  \gamma_B)} & 
    FEf \ar[d]^{\gamma_{Ef}} \\
    E'Gf \ar[r]_{\beta_f} & GEf\rlap{ ,}
  }
$$
and such that the induced $(\alpha, 1) \colon R'F^\atwo \to F^\atwo R$
are respectively a lax monad morphism $\mathsf R \to \mathsf R'$ and a
lax comonad morphism\footnote{In the terminology
  of~\cite{Street1972The-formal} a \emph{monad functor} and a
  \emph{comonad opfunctor}.} $\mathsf L \to \mathsf L'$ over ${F^\atwo
\colon \C^\atwo \to \D^\atwo}$. 
A \emph{transformation} $(F, \alpha) \Rightarrow (G, \beta)$
between lax morphisms is a natural transformation $\gamma \colon F
\Rightarrow G$ rendering commutative the square above right for each
$f \colon A \to B$ in $\C$.  Algebraic weak factorisation systems, lax morphisms and
transformations form a $2$-category $\Lax $.

Dually, there is a
2-category $\Oplax$ whose $1$-cells are the \emph{oplax morphisms} of
\awfs---for which the components $\alpha_{f}$ point in the opposite
direction to above---and whose $2$-cells are transformations between
them.
Both of these $2$-categories have an evident forgetful $2$-functor to
$\Cat$, and restricting $\Oplax \to \Cat$ to the fibre over some
category $\C$ yields the category $\Awfs{\C}$ of \awfs and \awfs
morphisms on $\C$; while restricting $\Lax \to \Cat$ to the fibre over
$\C$ yields the \emph{opposite} category $\Awfs{\C}^\op$.




Given \awfs $(\mathsf L, \mathsf R)$ and $(\mathsf L', \mathsf R')$ on
categories $\C$ and $\D$, and an adjunction ${F \dashv G \colon \C \to
  \D}$, there is a bijection between extensions of $G$ to a lax \awfs
morphism and extensions of $F$ to an oplax \awfs morphism obtained by
taking mates; this is the \awfs incarnation of the \emph{doctrinal
  adjunction} of~\cite{Kelly1974Doctrinal}. The functoriality of this
correspondence is usefully expressed as an identity-on-objects
isomorphism of $2$-categories $\Radj^{\co\op} \cong \Ladj$, where
$\Radj$ is defined identically to $\Lax$ except that its $1$-cells
come equipped with chosen left adjoints, and where $\Ladj$ is defined
from $\Oplax$ dually. Further details on doctrinal adjunction for
\awfs are in Section~2.10 of \cite{Bourke2014AWFS1}.

\subsection{Double-categorical semantics}
Given an \awfs $( \mathsf L, \mathsf R)$ on $\C$ we can consider the
Eilenberg--Moore categories $\Coalg{L}$ and $\Alg{R}$ of coalgebras
and algebras over $\C^{\atwo}$; these are thought of as providing the
respective left and right classes of the \awfs. An $\mathsf R$-algebra
$\alg f = (f, p) \colon A \to B$ is a morphism $f \colon A \to B$
equipped with algebra structure $p \colon Rf \to f$, while an algebra
morphism
\begin{equation}
\cd{
  A \ar[d]_{\alg f} \ar[r]^u & C \ar[d]^{\alg g} \\
  B \ar[r]_{v} & D}\label{eq:5}
\end{equation}
is a commuting square in $\C$ compatible with the algebra structures
on $\alg f$ and $\alg g$; similar notation and conventions will be
used for $\mathsf L$-coalgebras. The connection with the two classes
of a weak factorisation system is made on observing that
\begin{itemize}
\item Each map $f \colon A \to B$ has a factorisation $f = \fr f\cdot
  \fl f$ into a (cofree) $\mathsf L$-coalgebra followed by a (free)
  $\mathsf R$-algebra;\vskip0.25\baselineskip
\item Each commuting square~\eqref{eq:5} wherein $\alg f$ is an
  $\mathsf L$-coalgebra and $\alg g$ an $\mathsf R$-algebra has a
  \emph{canonical} diagonal filler: see Section~2.4 of
  \cite{Bourke2014AWFS1}.
\end{itemize}
It follows that each \awfs has an underlying weak factorisation system whose left
and right classes  are the retracts of $\mathsf
L$-coalgebras and $\mathsf R$-algebras respectively.

The right class of a weak factorisation system contains the identities
and is closed under composition. Correspondingly, in an \awfs, each
identity map $1_A \colon A \to A$ bears a \emph{unique} $\mathsf
R$-algebra structure $\alg 1_A \colon A \to A$; while if $\alg g
\colon A \to B$ and $\alg h \colon B \to C$ are $\mathsf R$-algebras,
then the composite underlying map $h \cdot g$ admits an $R$-algebra
structure $\alg h \cdot \alg g \colon A \to C$, \emph{uniquely}
determined by the requirement that the canonical lifts against $\alg h
\cdot \alg g$ should be obtained by first lifting against $\alg h$ and
then against $\alg g$.
The uniqueness just noted implies that composition of $\mathsf
R$-algebras is associative and unital, and that it is compatible with
$\mathsf R$-algebra maps: meaning that, for any $f \colon A \to B$,
there is an $\mathsf R$-algebra map $(f,f) \colon {\alg 1}_{A} \to
{\alg 1}_{B}$, and that if $(a,b) \colon \alg g \rightarrow \alg g'$
and $(b,c) \colon \alg h \to \alg h'$ are maps of $\mathsf
R$-algebras, then so too is $(a,c) \colon \alg h \cdot \alg g \to \alg
h' \cdot \alg g'$. 

We may express this composition and its associated coherences
concisely by saying that $\mathsf R$-algebras form a \emph{double
  category} $\DAlg{R}$ whose objects and horizontal arrows are those
of $\C$, and whose vertical arrows and squares are the $\mathsf
R$-algebras and the maps thereof. There is a forgetful double functor
$U^\mathsf R \colon \DAlg{R} \to \Sq{\C}$ into the double category of
commutative squares in $\C$, which, displayed as an internal functor
between internal categories in $\Cat$, is as on the left in:
\[
\cd[@C+2.2em@R-0.3em]{
  \Alg{R} \ar@<-5pt>[d]_d\ar@<5pt>[d]^c \ar[r]^{U^\mathsf R} &
  \C^\atwo
  \ar@<-5pt>[d]_d\ar@<5pt>[d]^c \\
  \ar[u]|i \C \ar[r]_1 & \C \ar[u]|i } \qquad \quad \qquad
\cd[@C+1.5em@R-0.3em]{
  \A_1 \ar@<-5pt>[d]_d\ar@<5pt>[d]^c \ar[r]^{V_1} &
  \C^\atwo
  \ar@<-5pt>[d]_d\ar@<5pt>[d]^c \\
  \ar[u]|i \A_0 \ar[r]_{V_0 = 1} & \C\rlap{ .} \ar[u]|i }\] 

Note that this internal functor has object component \emph{an identity}
and arrow component a \emph{faithful functor}. By a \emph{concrete
  double category over $\C$}, we mean a double category $\dcat A$ and
double functor $V \colon \mathbb A \to \Sq{\C}$, as on the right
above, with these same two
properties. 
In working with general concrete double categories, we may reuse our notation for
$\DAlg{R}$, writing $\alg f \colon A \to B$ to denote a vertical arrow
of $\mathbb A$ with $V(\alg f) = f \colon A \to B$, and
using~\eqref{eq:5} to depict a square of $\mathbb A$ which is sent by
$V$ to $(u,v) \colon f \to g$ in $\C^\atwo$; this is meaningful by
fidelity of $V_1$.

\begin{Ex}\label{ex:2}
Given a category $\C$, we write $\cat{SplEpi}(\C)$ for the category of
split epimorphisms therein: objects are pairs of a map $g \colon A \to
B$ of $\C$ together with a section $p$ of $g$, while morphisms $(g,p)
\to (h,q)$ are serially commuting diagrams:
\begin{equation}
\cd{A \ar@<-3pt>[d]_{g} \ar[r]^{u} & C \ar@<-3pt>[d]_{h} \\
B \ar@{-->}@<-3pt>[u]_{p} \ar[r]_{s} & D \rlap{ .}\ar@{-->}@<-3pt>[u]_{q}}\label{eq:6}
\end{equation}
Split epimorphisms compose---by composing the sections---so that we
have a double category $\dcat{SplEpi}(\C)$ which is concrete over $\C$ via the double
functor $\dcat{SplEpi}(\C) \to \Sq{\C}$ which forgets the sections.
\end{Ex}
\begin{Ex}
  \label{ex:3}
  A \emph{lali} (left adjoint left inverse) in a $2$-category $\C$ is
  a split epi $(g,p) \colon A \to B$ such that $g \dashv p$ with
  identity counit; note that this is a \emph{property} of $(g,p)$,
  rather than extra structure, since the unit of an adjunction is
  determined by the two functors and the counit. A morphism of lalis
  is simply a morphism of split epis; commutativity with the
  adjunction units is automatic. Since split epis and adjoints
  compose, so too do lalis; thus---writing $\C_0$ for the underlying
  category of $\C$---lalis in $\C$ form a concrete sub-double category
  $\dcat{Lali}(\C)$ of $\dcat{SplEpi}(\C_0)$. A \emph{retract
    equivalence} in $\C$ is a lali whose unit is invertible; these
  form a sub-double category of $\dcat{Lali}(\C)$.
\end{Ex}

\subsection{Pullback stability}
We now record the analogue for \awfs of the
pullback-stability of the right class of a weak factorisation
system---which is expressed in terms of a property of the double
category of algebras.

We call a functor $F \colon \A \to \C^\atwo$ a \emph{discrete
  pullback-fibration} if, for every $\alg g \in \A$ over
$g \in \C^\atwo$ and every pullback square $(h, k) \colon f \to g$,
there is a unique arrow $\phi \colon \alg f \to \alg g$ in $\A$ over
$(h,k)$, and this arrow is cartesian with respect to $F$; we call a
concrete double category $V \colon \dcat{A} \to \Sq{\C}$
\emph{pullback-stable} just when $V_{1} \colon \A_{1} \to \C^{\atwo}$
is a discrete pullback-fibration. In elementary terms, this asserts
that for any vertical map $\alg g \colon A \to B$ in $\mathbb A$ and
pullback square in $\C$ as on the right in:
$$\xy
(0,15)*+{A}="01";(15,15)*+{C}="11";
(0,0)*+{B}="00";(15,0)*+{D}="10";
{\ar_{v} "00"; "10"}; 
{\ar^{u} "01"; "11"}; 
{\ar_{\exists ! \alg f} "01"; "00"}; 
{\ar^{\alg g} "11"; "10"}; 
(35,7.5)*+{\longmapsto};
\endxy
\hspace{1cm}
\xy
(0,15)*+{A}="01";(15,15)*+{C}="11";
(0,0)*+{B}="00";(15,0)*+{D}="10";
{\ar_{v} "00"; "10"}; 
{\ar^{u} "01"; "11"}; 
{\ar_{f} "01"; "00"}; 
{\ar^{g} "11"; "10"}; 
\endxy$$
there exists a unique vertical map $\alg f$ over $f$ making $(u,v)$
into a cartesian square of $\mathbb A$ as on the left. The
cartesianness expresses that $(u,v)$ \emph{detects} squares of
$\mathbb A$---meaning that for any vertical arrow $\alg h$, a
commutative square $(r,s) \colon h \to f$ in $\C$ lifts to a square
$(r,s) \colon \alg h \to \alg f$ of $\mathbb A$ just when the
composite $(ru, tv) \colon h \to g$ lifts to one
$(ru,tv) \colon \alg h \to \alg g$. It follows
from~\cite[Proposition~8]{Bourke2014AWFS1} that the concrete double
category of algebras of an \awfs is always pullback-stable.

\subsection{Structure and semantics}
Each lax morphism $(F, \alpha) \colon (\C, \mathsf L, \mathsf R) \to
(\C', \mathsf L', \mathsf R')$ of \awfs induces a lifting of $\Sq{F}
\colon \Sq{\C} \to \Sq{\D}$ to a morphism of concrete double
categories as to the left in
\begin{equation}
\cd{
\DAlg{R} \ar[d]_{U^\mathsf R} \ar@{.>}[r]^{\bar F} & \DAlg{R'}
\ar[d]^{U^\mathsf{R'}} \\
\Sq{\C} \ar[r]_{\Sq{F}} & \Sq{\D}
} \qquad \qquad
\cd[@C+2.5em]{
\DAlg{R} \ar[d]_{U^\mathsf R} \ar@/^0.8em/@{.>}[r]^{\bar F} \ar@/_0.8em/@{.>}[r]_{\bar G} \dtwocell{r}{\bar{\alpha}}& \DAlg{R'}
\ar[d]^{U^\mathsf{R'}} \\
\Sq{\C} \ar@/^0.8em/[r]^{\Sq{F}} \ar@/_0.8em/[r]_{\Sq{G}} \dtwocell{r}{\Sq{\alpha}}  & \Sq{\D}\rlap{ .}
}\label{eq:8}
\end{equation}
Similarly, each $2$-cell of $\Lax$ induces a lifted horizontal
transformation as on the right; and in this way we obtain the
\emph{semantics 2-functor}
$$\DAlg{(\thg)} \colon \Lax \to \Dbl^{\atwo}\rlap{ .}$$
The following result, which combines Proposition~2 and Theorem~6
of~\cite{Bourke2014AWFS1}, shows that \awfs and their morphisms may be
characterised purely in terms of this double-categorical algebra
semantics. In its statement, a concrete double category
$\dcat{A} \to \Sq{\C}$ is said to be \emph{right-connected} just when
each vertical arrow $\alg f \colon A \to B$ of $\mathbb A$ can be
completed to a square:
\begin{equation}
\cd{ A\ar[r]^{f} \ar[d]_{\alg f} & B \ar[d]^{\alg 1_B} \\
B \ar[r]^{1_B} & B\rlap{ .}}
\label{eq:9}
\end{equation}

\begin{Thm}\label{thm:characterise}
  The 2-functor $\DAlg{(\thg)} \colon \Lax \to \Dbl^{\atwo}$ is
  $2$-fully faithful, and the 
  concrete $V \colon \dcat A \to \Sq{\C}$ is in its essential image
  just when:\vskip0.25\baselineskip
\begin{enumerate}[(i)]
\item The functor $V_1 \colon \A_{1} \to \C^{\atwo}$ on vertical arrows and squares is strictly monadic;
\item $\dcat{A} \to \Sq{\C}$ is right-connected.
\end{enumerate}
\end{Thm}
\begin{Rk}
  For the most part, the means by which Theorem~\ref{thm:characterise}
  was established in~\cite{Bourke2014AWFS1} will not concern us;
  however, there is one key point we will require later. Lemma~1 of
  \cite{Bourke2014AWFS1} asserts that, for any \awfs
  $(\mathsf L,\mathsf R)$ on $\C$, the commutative square
  \begin{equation}
    \cd[@-0.5em]{
      Ef \ar[dd]_{\fr {f}} \ar[r]^{\c f} & E\l f
      \ar[d]^{\fr {\l f}} \\
      & Ef \ar[d]^{\fr f} \\
      B \ar[r]_{1} & B\rlap{ .}
    }\label{eq:13}
  \end{equation}
  is a square in the double category $\DAlg{R}$, thus a morphism of
  $\mathsf R$-algebras. This provides the means by which the
  \emph{comultiplication} of the comonad $\mathsf L$ may be recovered
  from the \emph{double categorical} structure of $\DAlg{R}$---as by
  freeness of $\fr f$, the map
  $(\Delta_{f},1) \colon \fr f \to \fr f \cdot \fr Lf$ is the unique
  $\mathsf R$-algebra map whose precomposition with the unit
  $f \to \r f$ is $(LLf,1) \colon f \to Rf \cdot RLf$.\label{rk:2}
\end{Rk}
%
%


\begin{Ex}
  For any category $\C$, the concrete double category
  $\dcat{SplEpi}(\C)$ of Example~\ref{ex:2} is 
  right-connected; whence by Theorem~\ref{thm:characterise}, it will
  be the double category of algebras of an \awfs on $\C$ whenever $U
  \colon \cat{SplEpi}(\C) \to \C^\atwo$ is strictly monadic. We may
  identify $U$ with $\C^j \colon \C^\Ss \to \C^\atwo$, where $\Ss$ is
  the \emph{free split epi}:
\begin{equation}\label{eq:splitEpi}
\cd{
1 \ar[r]^{m} \ar[dr]_{1} & 0 \ar[d]_{e} \ar[dr]^{me} \\
& 1 \ar[r]^{m} & 0}
\end{equation}
and where $j \colon \atwo \to \Ss$ is the evident inclusion. Thus $U$
strictly creates colimits, and so will be strictly monadic whenever
$\C$ is cocomplete enough to admit left Kan extensions along~$j$.
Using the Kan extension formula one finds that only binary coproducts
are required; the free split epi $\fr f$ on $f \colon A \to B$ is
$\spn{f,1} \colon A+B \to B$ with section $\iota_B$.
\end{Ex}
\begin{Ex}
  \label{ex:4}
  For any $2$-category $\C$, the concrete double category
  $\dcat{Lali}(\C)$ of Example~\ref{ex:2} inherits right-connectedness
  from $\dcat{SplEpi}(\C_0)$; whence by
  Theorem~\ref{thm:characterise}, it will be the double category of
  algebras of an \awfs on $\C$ whenever $U \colon \cat{Lali}(\C) \to
  (\C_0)^{\atwo}$ is strictly monadic. Now, we may identify $U$ with
  restriction
\[
\cat{2}\text-\Cat(j, \C) \colon \cat{2}\text-\Cat(\L, \C) \to \cat{2}\text-\Cat(\atwo, \C)
\]
along the inclusion $j \colon \atwo \to \L$ of $\atwo$ into the
\emph{free lali} $\L$---which has the same underlying category as the
free split epimorphism $\Ss$ of \eqref{eq:splitEpi} and a single
non-trivial 2-cell $1\Rightarrow me$; so as before, monadicity obtains
whenever $\C$ is cocomplete enough to admit left Kan extensions along
$j$. In this case the necessary colimits are \emph{oplax colimits} of
arrows; see~\cite[Section~4.2]{Bourke2014AWFS1}. On inverting the
non-trivial $2$-cell of $\L$, it becomes the free \emph{retract
  equivalence} $\L^g$, and from this, we obtain for any sufficiently
cocomplete $2$-category $\C$ the \awfs on $\C_0$ whose algebras are retract equivalences.
\end{Ex}

\begin{Ex}
  \label{sec:projective}
  A basic way of obtaining new weak factorisation systems from old is
  by \emph{projective lifting} along a functor. In the algebraic
  context, the construction of projective liftings is simplified by
  Theorem~\ref{thm:characterise}. Given an \awfs $(\mathsf L, \mathsf
  R)$ on $\D$ and functor $F \colon \C \to \D$ one can form the
  pullback of double categories to the left in:
  \begin{equation}
    \cd[@C+0.5em]{
      \dcat A \pushoutcorner \ar[r] \ar[d]_{V} & \DAlg{R} \ar[d]^{U^{\mathsf
          R}}\\
      \Sq{\C} \ar[r]^{\Sq{F}} & \Sq{\D}} \qquad \qquad \qquad
    \cd{
      \A_1 \pushoutcorner\ar[r] \ar[d]_{V_1} & \Alg{R} \ar[d]^{U^{\mathsf
          R}}\\
      \C^\atwo \ar[r]^{F^\atwo} & \D^\atwo\rlap{ .}
    }\label{eq:39}
  \end{equation}
  It is easily verified that the concrete
  $V \colon \mathbb A \to \dcat{Sq}(\C)$ so obtained satisfies all the
  hypotheses of Theorem~\ref{thm:characterise} except possibly for the
  existence of a left adjoint to $V_1 \colon \A_1 \to \C^\atwo$, as on
  the right; so whenever this adjoint can be shown to exist, 
  $\mathbb A$ will comprise the algebra double category of an \awfs
  $(\mathsf L', \mathsf R')$ on $\C$, the \emph{projective lifting} of
  $(\mathsf L, \mathsf R)$ along $F$. Conditions ensuring the
  existence of this adjoint are described in Propositions~13 and~14 of
  \cite{Bourke2014AWFS1}.\label{ex:5}
\end{Ex}
  


\subsection{A reformulation of right-connectedness}
\label{sec:intr-concr}
As should be clear from Theorem~\ref{thm:characterise} above,
right-connected concrete double categories play an important role in
the theory of \awfs. So far, we have described such double categories
and the maps between them in terms of a full sub-$2$-category of
$\Dbl^\atwo$---but in fact we can do better. The
right-connectedness of the concrete $V \colon \dcat A \to \Sq{\C}$ is
easily seen to be equivalent to the property that the codomain functor
$c \colon \A_1 \to \A_0$ is left adjoint left inverse for the
identities functor $i \colon \A_0 \to \A_1$; we will say that $\mathbb
A$ is \emph{right-connected} if this is the case. Similarly, the
fidelity of $V_1 \colon \A_1 \to \C^\atwo$ is equivalent to $\mathbb
A$ being \emph{locally preordered}, in the sense that squares are
determined by their boundaries.

In fact any right-connected locally preordered $\dcat A$ arises in this
way; the corresponding $V \colon \mathbb A \to \Sq{\C}$ has $\C =
\A_0$, and $V_1 \colon \A_1 \to \A_0^\atwo$ the functor classifying
the composite natural transformation $d \eta \colon d \Rightarrow c
\colon \A_1 \to \A_0$.
If $\mathbb B$ is also right-connected and locally preordered, then any
double functor $F \colon \mathbb A \to \mathbb B$ will as in
Example~\ref{ex:3} commute with the units of the adjunctions
$c_\mathbb A \dashv i_\mathbb A$ and $c_\mathbb B \dashv i_\mathbb B$;
whence $F$ extends to a concrete double functor
between the associated concrete double categories. The same argument
pertains to $2$-cells, giving that:
\begin{Lemma}
  \label{lem:1}
  The domain $2$-functor $\Dbl^\atwo \to \Dbl$ restricts to a
  $2$-equivalence between the full sub-$2$-category of $\Dbl^\atwo$ on
  the right-connected concrete double categories, and the full
  sub-$2$-category of $\Dbl$ on the right-connected locally preordered
  double categories.
\end{Lemma}

\section{Categories of weak maps}
We now turn to the central object of study of this paper: the
categories of \emph{weak maps} associated to an algebraic weak
factorisation system. It turns out that these may be defined in a
number of different ways---the first and simplest of which is given
with reference to \emph{cofibrant replacement} comonads or
\emph{fibrant replacement} monads.
\subsection{Weak maps as Kleisli categories}
\label{sec:fibr-cofibr-objects}
If $(\mathsf L, \mathsf R)$ is an \awfs on $\C$, then its comonad
$\mathsf L$ restricts to a comonad on each coslice category $X / \C$;
in particular, if $X = 0$ is initial in $\C$, then we obtain a comonad
$\cof$ on $0 / \C \cong \C$, which we may call the \emph{cofibrant
  replacement comonad} of $(\mathsf L, \mathsf R)$. 
We define the category of \emph{left weak maps} $\Wkl$ to be the
co-Kleisli category of $\mathsf Q$.
Dually, if $\C$ has a terminal object, then the monad $\mathsf R$ on
$\C^\atwo$ induces a \emph{fibrant replacement monad} on $\C$,
whose Kleisli category is the category $\Wkr$ of
\emph{right weak maps}. 

\begin{Exs}
  \label{ex:1}
\begin{enumerate}[(i)]
\item Let $\C$ be any $2$-category admitting the \awfs for lalis of
  Example~\ref{ex:3} above. For each $X \in \C$, this \awfs induces a
  \emph{coslice} \awfs on $X/\C$ by projectively lifting, as in
  Example~\ref{sec:projective}, along the forgetful functor $X/\C \to
  \C$. Direct calculation shows that the category of left weak maps
  for this \awfs is the lax coslice $X /\!/ \C$, wherein morphisms
  $(A,a) \rightsquigarrow (B,b)$ are lax triangles
\[
\cd{ & X \ar[dl]_a \ar[dr]^b \ltwocell{d}{\phi} \\A \ar[rr]_f & &
  B\rlap{ .} }\] By starting from the \awfs for retract equivalences
rather than that for lalis, we obtain instead the \emph{pseudo}
coslice category. \vskip0.25\baselineskip
\item In~\cite{Garner2010Homomorphisms} was constructed an \awfs on
  $\cat{Bicat}_s$, the category of small bicategories and strict
  morphisms---those preserving composition and identities on the nose---whose associated
  left weak maps are the \emph{homomorphisms}
  of bicategories---those preserving composition and identities up to coherent isomorphism.
  \cite{Garner2010Homomorphisms} also exhibits trihomomorphisms
  between tricategories as weak maps, and, more generally, shows that
  for any \emph{globular operad} $\O$ in the sense
  of~\cite{Batanin2011Algebras}, there is an \awfs on the category of
  $\O$-algebras and strict $\O$-algebra maps whose category of left
  weak maps provides a suitable notion of ``weak morphism'' of
  $\O$-algebras.

\vskip0.25\baselineskip
\item Let $\C$ be a category with pullbacks, and $\M$ a class of
  monomorphisms therein which contains all isomorphisms and is stable
  under composition and pullback. Suppose moreover that there is an
  \emph{classifying $\M$-map}---an $\M$-map $t \colon 1
  \rightarrowtail U$ which is terminal in the category of $\cat
  M_\mathrm{pb}$ of $\M$-maps and pullback squares---and that $t$ is
  exponentiable in $\C$. Under this assumption, it was shown
  in~\cite[Section~4.4]{Bourke2014AWFS1} that there is an \awfs on $\C$
  whose category of $\mathsf L$-coalgebras is $\cat{M}_\mathrm{pb}$,
  and whose fibrant replacement monad is the \emph{partial $\M$-map
    classifier monad} of~\cite[Chapter~3]{Rosolini1986Continuity}:
  \[
  \fib = \C \xrightarrow{\Pi_t} \C / U \xrightarrow{\Sigma_U}
  \C\rlap{ .}
  \]
  The associated category of right weak maps is the \emph{$\M$-partial
    map category}, whose morphisms $A \rightsquigarrow B$ are
  isomorphism-classes of spans $A \leftarrowtail A' \to B$ with left
  leg in $\M$, and whose composition is by pullback.\vskip0.25\baselineskip
\item In Section~\ref{sec:two-dimens-monad} we will see that, if
  $\mathsf T$ is an accessible $2$-monad on a complete and cocomplete
  $2$-category $\C$, then there are \awfs on $\Alg{T}_s$, the category
  of $\mathsf T$-algebras and \emph{strict} algebra morphisms, whose
  categories of left weak maps have as morphisms the \emph{pseudo} or
  \emph{lax} algebra morphisms.
  \vskip0.25\baselineskip
\item In Section~\ref{sec:homological} we will see that, if $\C$ is a
  cocomplete dg-category and $\mathsf T$ a sufficiently cocontinuous
  dg-monad thereon, then there is an awfs on $\TAlgs$, the category of
  $\mathsf T$-algebras and strict morphisms whose category of weak
  maps is the category of \emph{homotopy-coherent} $\mathsf T$-algebra
  morphisms. For example, if $\A$ is a small dg-category and $\D$ a
  cocomplete one, then there is an \awfs on the functor category
  $[\A, \D]$ for which the left weak maps are \emph{homotopy-coherent
    natural transformations} between dg-functors in the sense
  of~\cite[\S3]{Tamarkin2007What}.
\end{enumerate}
\end{Exs}

\subsection{The universal property of the category of weak maps}
\label{sec:univ-prop-categ}
Our first main result expresses that the category of left weak maps arises
by ``freely splitting each $\mathsf R$-map in $\C$''. Of course, there
is a dual result for right weak maps---but henceforth we 
leave it to the reader to formulate such dual cases.
\begin{Thm}
\label{thm:WeakMaps}
Let $(\mathsf L, \mathsf R)$ be an \awfs on a category $\C$ with an
initial object. There are isomorphisms of categories
\begin{equation}
\Cat(\Wkl, \D) \cong \Dbl(\DAlg{R}, \dcat{SplEpi}(\D))\rlap{ ,}\label{eq:12}
\end{equation}
$2$-natural in $\D$, that mediate between extensions of a functor $F
\colon \C \to \D$ through the co-Kleisli category $\Wkl$
as on the left below, and liftings of $F$ to a concrete double functor
as on the right.
\begin{equation}
  \cd[@C-0.5em]{
    & \Wkl \ar@{.>}[dr]^{G} & & & \DAlg{R} \ar[d] \ar@{.>}[r]^-{H} & \dcat{SplEpi}(\D) \ar[d] \\ 
    \C \ar[ur]^{\text{\textnormal{cofree}}} \ar[rr]_{F} & & \D & & \Sq{\C}
    \ar[r]_-{\Sq{F}} & \Sq{\D}\rlap{ .} 
  }\label{eq:23}
\end{equation}
\end{Thm}
\begin{proof}
  It is easy to see that $\dcat{SplEpi}(\thg) \colon \Cat \to \Dbl$
  preserves all $2$-dimensional limits; in particular we have
  $\dcat{SplEpi}(\D^\atwo) \cong \dcat{SplEpi}(\D)^\atwo$, so that to
  obtain isomorphisms of categories~\eqref{eq:12} it suffices to
  exhibit ones of underlying \emph{sets}. As $\DAlg{R}$ and
  $\dcat{SplEpi}(\D)$ are right-connected and locally preordered, this is
  equivalent by Lemma~\ref{lem:1} to giving a bijection, natural in
  $\D$, between extensions $G$ and liftings $H$ as
  in~\eqref{eq:23}.
  
  Suppose first that we are given a lifting of $F$ to a double
  functor as on the right. For each $B \in \B$, let $!_B \colon 0 \to B$
  denote the unique map; then for any $\mathsf R$-algebra $\alg f
  \colon A \to B$, the map $(!_A, 1_B) \colon !_B \to f$ in $\C^\atwo$
  induces an $\mathsf R$-algebra morphism as on the left below.
  Applying~\eqref{eq:23} yields a morphism of split epis in $\D$ as on
  the right (observing that the underlying map of $\fr !_B$ is the
  counit $\epsilon_B \colon QB \to B$).
  \begin{equation}
\cd{
QB \ar[d]_{\fr{!_B}} \ar[r]^{\phi_{\alg f}} & A \ar[d]^{\alg f}\\
B \ar[r]_{1} & B
} \qquad \qquad
\cd{
FQB \ar@<-3pt>[d]_{F\epsilon_B} \ar@<3pt>@{<--}[d]^{\alpha_B}
\ar[r]^{F\phi_{\alg f}} & FA \ar@<-3pt>[d]_{F f}
\ar@<3pt>@{<--}[d]^{s}\\
FB \ar[r]_{1} & FB
}
\label{eq:27}
\end{equation}
By the upwards commutativity we have $s = F \phi_{\alg f} \cdot
\alpha_B$, so that the lifting~\eqref{eq:23} is uniquely determined by
giving the $\alpha_B$'s. These maps are the components of a
natural transformation $\alpha \colon F \to FQ$ with $F \epsilon \cdot
\alpha = 1 \colon F \to FQ \to F$; and since applying~\eqref{eq:23} to
the algebra square~\eqref{eq:13} below left yields the one below right:
\begin{equation}
  \label{eq:20}
  \cd[@R-0.1em@-0.2em]{
QB \ar[dd]_{\fr !_B} \ar[r]^{\Delta_{B}} & QQB 
\ar[d]^{\fr{\l !_B}= \fr !_{QB}} \\
& QB \ar[d]^{\fr !_B } \\
B \ar[r]_1 & B} \qquad
  \cd[@R-0.1em@-0.2em]{
FQB \ar@<-3pt>[dd]_{F \epsilon_B} \ar@<3pt>@{<--}[dd]^{\alpha_B} \ar[r]^{F\Delta_{B}} & FQQB 
\ar@<-3pt>[d]_{F\epsilon_{QB}} \ar@<3pt>@{<--}[d]^{\alpha_{QB}} \\
& FQB \ar@<-3pt>[d]_{F \epsilon_B} \ar@<3pt>@{<--}[d]^{\alpha_B} \\
FB \ar[r]_1 & FB\rlap{ ,}}
\end{equation}
we conclude that also $\alpha Q \cdot \alpha = F \Delta \cdot \alpha
\colon F \to FQ \to FQQ$. Thus $\alpha \colon F \to FQ$ exhibits $F$
as a coalgebra for the comonad $(\thg) \cdot \mathsf Q$ on the functor
category $[\C, \D]$; and by~\cite[\S 5]{Street1972The-formal}, to give
such a coalgebra structure is equally to give an extension of $F$ to a
functor $\Wkl = \Kl{Q} \to \D$. The process just described is clearly natural
in $\D$; so to complete the proof, it suffices to show
that any coalgebra $\alpha \colon F \to FQ$ arises from a double
functor in this way.

For this, given a coalgebra $\alpha$, we must show that the
lifting~\eqref{eq:23} which sends an $\mathsf R$-algebra $\alg f
\colon A \to B$ to the split epi $(Ff, F\phi_{\alg f} \cdot \alpha_B)
\colon FA \to FB$ is well-defined. The only point that needs checking
is that composition of algebras is preserved; thus given also $\alg g
\colon B \to C$, we must show that $F\phi_{\alg f} \cdot \alpha_B
\cdot F\phi_{\alg g} \cdot \alpha_C = F\phi_{\alg g \alg f} \cdot
\alpha_C$. The left-hand side is equal to $F\phi_{\alg f} \cdot
FQ\phi_{\alg g} \cdot \alpha_{QC} \cdot \alpha_C = F\phi_{\alg f}
\cdot FQ\phi_{\alg g} \cdot F \Delta_C \cdot \alpha_C$ by naturality
 and the comultiplication axiom for $\alpha$. So it suffices to show
 that $\phi_{\alg f} \cdot Q \phi_{\alg g} \cdot \Delta_C = \phi_{\alg
   g \alg f}$. For this, we consider the following diagram.
 \begin{equation*}
  \cd{
QC \ar[dd]_{\fr !_C} \ar[r]^{\Delta_{C}} & QQC \ar[r]^{Q\phi_{\alg g}}
\ar[d]^{\fr{\l !_C}} &
QB \ar[r]^{\phi_{\alg f}} \ar[d]^{\fr !_B} & A \ar[d]^{\alg f} \\
& QC \ar[d]^{\fr !_C} \ar[r]^{\phi_{\alg g}} & B \ar[d]^{\alg g}
\ar[r]_1 & B \ar[d]^{\alg g} \\
C \ar[r]_1 & C \ar[r]_1 & C \ar[r]_1 & C\rlap{ .}}
\end{equation*}
Each small square is an $\mathsf R$-algebra map, whence the large
rectangle is too; but since this rectangle precomposes with the unit
$!_C \to \r !_C$ to yield $(!_A, 1_C) \colon !_C \to gf$, it must by
freeness of $\fr !_C$ be the map $(\phi_{\alg g \alg f}, 1)$.
Comparing domain-components yields that $\phi_{\alg f} \cdot Q
\phi_{\alg g} \cdot \Delta_C = \phi_{\alg g \alg f}$ as required.
\end{proof}

\subsection{Weak maps as weighted colimits}
\label{sec:weak-maps-as}
In order to define the category of left weak maps above, we were
forced to assume the existence of an initial object; and while this is
scarcely a restriction in practice, it is nonetheless a little
inelegant. Theorem~\ref{thm:WeakMaps} suggests a way of removing this
restriction: we \emph{redefine} the category of left weak maps of an
\awfs $(\mathsf L, \mathsf R)$ to be any category which, as
in~\eqref{eq:12}, represents the $2$-functor $\Dbl(\DAlg{R},
\dcat{SplEpi}(\thg)) \colon \Cat \to \Cat$. We shall adopt this
broader definition henceforth; the problem now becomes one of
determining when $\Wkl$ exists. Theorem~\ref{thm:WeakMaps} tells us
that this is so whenever $\C$ has an initial object; what we now show
is that this is in fact \emph{always} so: $\Wkl$ can be computed,
under no restrictions on $\C$, as a certain weighted colimit in
$\Cat$. 

To see this, let $\Delta_2$ be the subcategory
$$\xy
(0,0)*+{[0]}="00"; (30,0)*+{[1]}="10";(60,0)*+{[2]}="20";
{\ar@<1.2ex>^{\delta_{1}} "00"; "10"}; 
{\ar@<0ex>|{\sigma_{0}} "10"; "00"}; 
{\ar@<-1.2ex>_{\delta_{0}} "00"; "10"}; 
{\ar@<1.2ex>^{\delta_{2}} "10"; "20"}; 
{\ar@<0ex>|{\delta_{1}} "10"; "20"}; 
{\ar@<-1.2ex>_{\delta_{0}} "10"; "20"};
\endxy$$
of the simplicial category, where
$[n] = \{0 \leqslant \dots \leqslant n\}$ and $\delta$ and $\sigma$
are the usual coface and codegeneracy operators. Identifying each
internal category in $\Cat$ with its truncated nerve gives a full
embedding $\Dbl \to [\Delta_2^\op, \Cat]$ of $2$-categories; it
follows that we may 
view the
defining isomorphism~\eqref{eq:12} as one
\[
\Cat(\Wkl, \D) \cong [\Delta_2^\op, \Cat](\DAlg{R},
\dcat{SplEpi}(\D))\rlap{ .}
\]
Now, the $2$-functor $\dcat{SplEpi}(\thg) \colon \Cat \to
[\Delta_2^\op, \Cat]$ sends a category $\D$ to the truncated
simplicial diagram as to the left of
\[
\cd[@C+0.2em]{ \D \ar[r]|i \ar@{<-}@<1.2ex>[r]^d
  \ar@{<-}@<-1.2ex>[r]_c & \D^\Ss \ar@{<-}[r]|(0.4)m
  \ar@<1.2ex>@{<-}[r]^(0.4)p \ar@<-1.2ex>@{<-}[r]_(0.4)q & \D^\Ss
  \times_\D \D^\Ss } \qquad \qquad \cd[@C+0.2em]{ \aone \ar@{<-}[r]|i
  \ar@<1.2ex>[r]^d \ar@<-1.2ex>[r]_c & \Ss \ar[r]|(0.4)m
  \ar@<1.2ex>[r]^(0.4)p \ar@<-1.2ex>[r]_(0.4)q & \Ss +_\aone \Ss
  \rlap{ } }\] where $\Ss$ is the free split epi as
in~\eqref{eq:splitEpi}. We may see this diagram as induced by homming
into $\D$ from the cosimplicial diagram $S \colon \Delta_2 \to \Cat$
right above; in other words, we have $\dcat{SplEpi}(\thg) \cong
\Cat(S,1) \colon \Cat \to [ \Delta_2^\op, \Cat]$. Since $\Cat(S,1)$ has a left adjoint sending $X
\in [\Delta_2^\op, \Cat]$ to the weighted colimit $S \star X$, we
conclude that we have a $2$-representation
\begin{equation}\label{eq:splitting}
\Cat(S \star \DAlg{R}, \D) \cong \Dbl(\DAlg{R},
\dcat{SplEpi}(\D))\rlap{ ,}
\end{equation}
so that the category of left weak maps of the \awfs $(\mathsf L,
\mathsf R)$ may be constructed under no assumptions on $\C$ as the
weighted colimit $S \star \DAlg{R}$.

Unfolding the description, we find that for a right-connected double
category $\dcat A$, the colimit $S \star \dcat A$ may be obtained by
first taking a coinserter $v \colon \A_0 \to \B$ of $c, d \colon \A_1
\rightrightarrows \A_0$, with universal $2$-cell $\theta \colon vc
\Rightarrow vd$, and then taking three coequifiers: one making $\theta$
 into a section of $vd\eta \colon vd \Rightarrow vc$, and two imposing the cocycle
conditions $\theta m = \theta p \cdot \theta q \colon vcq \Rightarrow
vdp$ and $\theta i = 1$.
Thus $S \star \mathbb A$ is the category obtained from $\A_0$ by
freely adjoining a morphism ${\alg a}^\ast \colon X \to A$ for each
$\alg a \colon A \to X$ in $\mathbb A$ satisfying $a \cdot {\alg
  a}^\ast = 1$, $(\alg b \cdot \alg a)^\ast = {\alg a}^\ast \cdot
{\alg b}^\ast$ and ${\alg 1}^\ast = 1$, and $u \cdot {\alg a}^\ast = {\alg b}^\ast \cdot v$ for
all $(u,v) \colon \alg a \to \alg b$ in $\mathbb A$.

There is some redundancy in the above description. Right connectedness
of $\mathbb A$ ensures that any $\theta \colon vc \Rightarrow vd$ is
automatically a section of $vd\eta$ so that, in fact, only the two
cocycle conditions need to be imposed. The 2-categorically minded
reader will observe that these two conditions alone present
$S \star \dcat A$ as the \emph{codescent object} of the
\emph{opposite} double category $\mathbb A^{\op}$---in which $d$ and
$c$ are interchanged. Although codescent objects are the better known
colimit, the universal property expressed
in~\eqref{eq:splitting} is the more useful one in the context of
\awfs.

Note that applying this construction to a locally small category may
yield one that is no longer locally small; by contrast, the
construction of categories of weak maps via cofibrant replacement
always preserves local smallness. This is analogous to the fact that
that the homotopy category of a locally small model category is again
locally small, while the localisation of a category at an arbitrary
collection of morphisms need not be so.

\subsection{Weak maps as spans}
We have already mentioned the analogy between the weak map
construction, which freely adds sections for $\mathsf R$-algebras, and
the better known construction of the \emph{localisation} $\C[\W^{-1}]$
of a category $\C$ at a class of morphisms $\W$: here, rather than
freely adding a section to each map, we add an inverse. Under certain
hypotheses~\cite{Gabriel1967Calculus}, the morphisms of $\C[\W^{-1}]$
from $A$ to $B$ can be viewed as equivalence classes of spans $A
\leftarrow X \to B$ whose left leg belongs to $\W$. We now show that
under similar hypotheses, the category of left weak maps can
also be described in such terms.

Let $\C$ be a category with pullbacks and $U\colon \mathbb A \to
\Sq{\C}$ a pullback-stable concrete double category over $\C$. We
begin by forming a bicategory $\cat{Span}(\dcat{A})$ that enhances the usual
bicategory of spans in $\C$. Objects are those of $\C$, while
morphisms from $A$ to $B$ are \emph{$\mathbb{A}$-spans}: that is, spans
\begin{equation}
\cd{A & X \ar[l]_{\b{a}} \ar[r]^{f} & B}
\label{eq:1}
\end{equation}
 in $\C$ whose left leg has the structure of a vertical arrow of
$\dcat A$. The 2-cells $(\alg a, f) \to (\alg b, g)$ are
\emph{$\mathbb{A}$-span maps}: that is, span maps
\begin{equation}
\cd{& X \ar[dl]_{\b{a}} \ar[d]^{r} \ar[dr]^{f}\\
A & Y\ar[l]^{\b{b}} \ar[r] _{g} & B}
\label{eq:2}
\end{equation}
in $\C$ for which $(r,1) \colon \b{a} \to \b{b}$ is a square of $\dcat
A$. The identity morphism at $A$ is $(\b{1},1) \colon A \to A$; while
the composite of $(\b{a},f) \colon A \to B$ and $(\b{b},g) \colon B
\to C$ is obtained by first composing the underlying spans in $\C$ by
pullback:
\begin{equation}
\xy
(12,12)*+{X \times_{B} Y}="f0"; 
(0,0)*+{X}="a0"; (-12,-12)*+{A}="b0";(12,-12)*+{B}="c0";
{\ar_{\b{a}} "a0"; "b0"}; 
{\ar^{f} "a0"; "c0"}; 
(24,0)*+{Y}="d0";(36,-12)*+{C}="e0";
{\ar_{\b{b}} "d0"; "c0"}; 
{\ar^{g} "d0"; "e0"}; 
{\ar_{\b{p}} "f0"; "a0"}; 
{\ar^{q} "f0"; "d0"}; 
\endxy
\label{eq:3}
\end{equation}
and then using the fact that $U \colon \A_{1} \to \C^{\atwo}$ is a
discrete pullback-fibration to induce a structure of vertical arrow on
$p$ as displayed; this now allows us to define the composite $(\alg b,
f) \cdot (\alg a, f)$ to be $(\b{a}\cdot \b{p},g\cdot q)$.
The remaining aspects of the bicategory structure---$2$-cell
composition, associativity and unit constraints, and coherence
axioms---are as for the usual bicategory of spans. The additional
facts to be established are that various span maps are in fact $\dcat
A$-span maps, and here one makes full use of the fact that $U \colon
\A_{1} \to \C^{\atwo}$ is a discrete pullback-fibration.

\newcommand{\Wk}{\C[\dcat A^*]}

Our present interest lies not in the above bicategory, but in a
quotient of it. Each bicategory $\C$ gives rise a category $\pi_{1}\C$
with the same objects as $\C$ and with $(\pi_{1}\C)(X,Y)$ the set of
connected components of the category $\C(X,Y)$; we define $\Wk$ to be
the category $\pi_{1}(\cat{Span}(\dcat{A}))$. So objects of $\Wk$ are
those of $\C$, while maps $A \to B$ are equivalence classes of
$\mathbb{A}$-spans, where $(\b{a},f) \sim (\b{b},g) \colon A \to B$
just when they can be connected by a zigzag of $\dcat A$-span
maps.\begin{footnote}{For those $\dcat A$ arising in practice, it is
    often the case that $V_1 \colon \A_{1} \to \C^{\atwo}$ creates
    pullbacks---for example when it is monadic. This ensures that each
    hom-category of $\cat{Span}(\dcat A)$ has pullbacks, from which it
    follows that that $(\b{a},f) \colon A \to B$ and
    $(\b{b},g) \colon A \to B$ are equal in $\Wk$ if and only if there
    exists a third $\mathbb{A}$-span $(\b{c},h) \colon A \to B$ and
    $\mathbb A$-span maps $(\b{c},h) \to (\b{a},f)$ and
    $(\b{c},h) \to (\b{b},g)$.}\end{footnote}
%
There is an evident functor $J \colon \C \to \Wk$ which is the
identity on objects and sends $f \colon A \to B$ to $[\b{1},f] \colon
A \to B$. Note that, as in the preceding section, the category
$\C[\dcat A^\ast]$ may not be locally small even if $\C$ is so.

The construction of $\C[\dcat A^\ast]$ and the proof of the
following result have their origin in the 2-categorical
constructions of Section~4.2 and Theorem~21 of
\cite{Bourke2014Two-dimensional}.

\begin{Thm}\label{thm:pullback}
  Let $\C$ be a category with pullbacks and $\dcat{A}$ a
  pullback-stable right-connected concrete double category over $\C$.
  There are $2$-natural isomorphisms of
  categories $$\Cat(\Wk,\D) \cong
  \Dbl(\dcat{A},\dcat{SplEpi}(\D))$$ exhibiting $\Wk$ as
  $S \star \dcat A$. These isomorphisms mediate between extensions of
  a functor $F \colon \C \to \D$ through $\Wk$ as on the left
  below, and liftings of $F$ to a concrete double functor as on the
  right:
\begin{equation}
  \cd{
    & \Wk \ar@{.>}[dr]^{G} & & & \dcat{A} \ar[d] \ar@{.>}[r]^-{H} & \dcat{SplEpi}(\D) \ar[d] \\     \C \ar[ur]^{J} \ar[rr]_{F} & & \D & & \Sq{\C}
    \ar[r]_-{\Sq{F}} & \Sq{\D}\rlap{ .}   }\label{eq:Ext}
\end{equation}
\end{Thm}
\begin{proof}
  Arguing as in Theorem~\ref{thm:WeakMaps}, it suffices to exhibit a
  natural bijection between extensions $G$ and liftings $H$ as
  in~\eqref{eq:Ext}. Suppose first that we are given a lifting of $F$
  to a double functor $H$: such is specified by the assignment to each
  $\b{a} \colon A \to B$ of a split epi $(Fa,{\b{a}}^\ast) \colon FA
  \to FB$ in $\D$ subject to functoriality
  axioms. 

  Given this, we must define an extension $G \colon \Wk
  \to \D$ as on the left. Of course, we take $GX = FX$ on objects. On
  morphisms, given $[\b{a},f] \colon A \leftarrow X \rightarrow B$ in
  $\Wk$, we have an arrow $Ff \cdot {\b{a}}^\ast \colon
  FA \to FX \to FB$ in $\D$; furthermore, for any $\mathbb{A}$-span
  map $r \colon (\b{a},f) \to (\b{b},g)$, the $\dcat A$-square $(1,r)
  \colon \alg a \to \alg b$ is sent to a morphism $(1,Fr) \colon
  (Fa,{\b{a}}^\ast) \to (Fb,{\b{b}}^\ast)$ of split epis. Thus the
  left diagram in
\begin{equation}
\begin{gathered}
\vcenter{\hbox{\xy
(0,0)*+{FA}="00";(15,15)*+{FX}="11"; (30,0)*+{FB}="20"; (15,0)*+{FY}="10";
{\ar^{{\b{a}}^\ast} "00"; "11"}; 
{\ar_{{\b{b}}^\ast} "00"; "10"}; 
{\ar^{Ff} "11"; "20"}; 
{\ar_{Fg} "10"; "20"}; 
{\ar^{Fr} "11"; "10"}; 
\endxy}}
\hspace{1.5cm}
\vcenter{\hbox{\xy
(12,24)*+{F(X \times_{B} Y)}="f0"; 
(0,12)*+{FX}="a0"; (-12,0)*+{FA}="b0";(12,0)*+{FB}="c0";
{\ar^{{\b{a}}^\ast} "b0"; "a0"}; 
{\ar^<<<{Ff} "a0"; "c0"}; 
(24,12)*+{FY}="d0";(36,0)*+{FC}="e0";
{\ar^{{\b{b}}^\ast} "c0"; "d0"}; 
{\ar^{Fg} "d0"; "e0"}; 
{\ar^<<<<{{\b{p}}^\ast} "a0"; "f0"}; 
{\ar^<<<<<<{Fq} "f0"; "d0"}; 
\endxy}}
\end{gathered}
\end{equation}
commutes, so that taking $G[\b{a},f] = Ff \cdot {\alg a}^\ast$ gives
a well-defined assignation on morphisms. Given $f \colon A \to B \in
\C$ we have $G[\b{1},f]=Ff \cdot \alg 1^\ast=Ff$ so that $G$ is indeed an
extension of $F$, and the same argument shows that it preserves
identities. As for binary composition, let $[\b{a},f] \colon A \to B$
and $[\b{b},g] \colon B \to C$ in $\Wk$ with composite
$[\b{a}\cdot \b{p},g\cdot q]$ as in~\eqref{eq:3}. Now $G[\alg b, g]
\cdot G[\alg a, f]$ and $G([\b b, g] \cdot [\b a, f])$ are the lower
and upper composites on the right above, so it suffices to show that
the inner square commutes---but as $(q,f) \colon \b{p} \to \b{b}$ in $\dcat
A$, its image under $H$ is a map of split epis $(Fp,{\b{p}}^\ast) \to
(Fb,{\b{b}}^\ast)$, whence the square commutes as required.

The passage from $H \colon \dcat{A} \to \dcat{SplEpi}(\D)$ to $G
\colon \Wk \to \D$ is clearly natural in $\D$; it is
also injective, as given $\b{a} \colon A \to B$ we have ${\b{a}}^\ast
= G[\b a, 1] \colon FB \to FA$. To complete the proof, it thus
suffices to show that each extension $G \colon \Wk \to
\D$ is induced by some $H \colon \dcat A \to \dcat{SplEpi}(\D)$. By
naturality, it is enough to show that $1 \colon \Wk \to
\Wk$ is induced by some $H \colon \dcat A \to
\dcat{SplEpi}(\Wk)$ lifting $J$.

We define $H$ on vertical morphisms by $\alg a \mapsto (Ja, [\alg a,
1]) = ([\alg 1, a], [\alg a, 1])$. So long as this is well-defined,
the corresponding extension $\Wk \to \Wk$
will send $[\b{a},f] \colon A \to B$ to $Jf \cdot [\b{a},1] = [1, \alg
f] \cdot [\b a, 1] = [\alg a, f]$ and so be the identity as required.
It remains to show well-definedness of $H$. By $\dcat A$'s
right-connectedness, $a \colon (\b{a},a) \to (\b{1},1)$
is an $\dcat{A}$-span map, whence $[\alg a, 1] \cdot [\alg 1, a] =
[\alg a, a] = 1$ so that $([\alg 1, a], [\alg a, 1])$ is indeed a
split epi in $\Wk$. Vertical functoriality of $H$
follows from the easy fact that $[\alg a, 1] \cdot [\alg b, 1] = [\alg
a \cdot \alg b, 1]$, and so it remains to prove that $H$ is
well-defined on squares.

To this end, let $(r,s) \colon \b{a} \to \b{b}$ in $\dcat A$; we must show that
$([\b{1},r],[\b{1},s])$ is a map of split epis $([\b{1},a],[\b{a},1])
\to ([\b{1},b],[\b{b},1])$, so that $[\b{1},r]
\cdot [\b{a},1] = [\b{b},1] \cdot [\b{1},s] \colon B \to C$ in
$\Wk$. The left-hand side composes to $[\b{a},r]
\colon B \to C$ whilst the right-hand is the composite $[\b{1} \cdot
\b{p},1 \cdot q]=[\b{p},q] \colon B \to C$ in
$$
\xy
(12,36)*+{A}="g0";
(12,24)*+{B \times_{D} C}="f0"; 
(0,12)*+{B}="a0"; (-12,0)*+{B}="b0";(12,0)*+{D}="c0";
{\ar_{\b{1}} "a0"; "b0"}; 
{\ar^{s} "a0"; "c0"}; 
(24,12)*+{C}="d0";(36,0)*+{C\rlap{ .}}="e0";
{\ar_{\b{b}} "d0"; "c0"}; 
{\ar^{1} "d0"; "e0"}; 
{\ar_{\b{p}} "f0"; "a0"}; 
{\ar^{q} "f0"; "d0"}; 
{\ar@{.>}^{k} "g0"; "f0"}; 
{\ar@/_1.5pc/_{\b{a}} "g0"; "a0"}; 
{\ar@/^1.5pc/^{r} "g0"; "d0"}; 
\endxy$$
There is a unique $k \colon A \to B \times_{D} C$ as displayed
with $p \cdot k = a$ and $q \cdot k =r$; and since the cartesian $(q,s)
\colon \b{p} \to \b{b}$ detects squares and $(r,s) \colon \b{a} \to
\b{b}$ is an $\dcat A$-square, it follows that $(k,1) \colon \b{a} \to
\b{p}$ is an $\dcat A$-square. Thus $k \colon (\b{a},r) \to
(\b{p},q)$ is a morphism of $\dcat{A}$-spans, and so $[\b{1},r] \cdot
[\b{a},1]=[\alg a, r] = [\alg p, q] = [\b{b},1] \cdot [\b{1},s]$ as
required.
\end{proof}
Applying this theorem in the situation where we have an
\awfs $(\mathsf L, \mathsf R)$ on a category $\C$ with pullbacks, we deduce
that we may construct the category $\Wkl$ of left weak maps
as $\C[\DAlg{R}^\ast]$. If $\C$ also has an initial
object, then our original construction of $\Wkl$ as $\Kl{\mathsf Q}$
also applies---and we conclude that:
\begin{Cor}
  If $(L,R)$ is an \awfs on a category $\C$ with an initial object and
  pullbacks, then there is a unique isomorphism $\Kl{Q} \cong
  \C[\DAlg{R}^\ast]$ commuting with the maps from $\C$.
\end{Cor}
Explicitly, this isomorphism identifies a map $f \colon QA \to B$ in
the Kleisli category with the equivalence class of $(\fr !_{A},f)
\colon A \leftarrow QA \to B$.
\begin{Rk}
  \label{rk:3}
  As well as being related to~\cite{Bourke2014Two-dimensional}, the
  construction and proof of Theorem~\ref{thm:pullback} are also
  closely related to parts of~\cite{Weber2014Codescent}. In Definition
  5.1.1 of \emph{ibid.}, Weber defines a \emph{crossed internal
    category} in $\cat{CAT}$ to be an internal category $\mathbb A$
  whose codomain map is a split opfibration, and whose identity and
  composition morphisms are strict maps of split opfibrations. In
  Theorem 5.4.4, he gives a construction of the codescent object of a
  crossed internal category which, after dualising appropriately, is
  more-or-less the construction of $\C[\mathbb A^\ast]$ given above.

  The source of this connection is as follows. If $\C$ is a category
  with pullbacks, and $\mathbb A$ is a pullback-stable and
  right-connected double category over $\C$, then its codomain map is
  a \emph{fibration}, and its identity and composition maps are
  pseudomorphisms of fibrations. Ignoring the inessential facts of the
  fibration being non-split and the identity and composition maps
  being non-strict, it follows that the double category
  $\mathbb A^\co = (\A_1)^\op \rightrightarrows (\C^\atwo)^\op$ is a
  crossed internal category. Now, as we remarked in
  Section~\ref{sec:weak-maps-as} above, the universal property of
  $\C[\mathbb A^\ast]$ as the weighted colimit $S \star \dcat A$
  equally well makes it the codescent object of $\mathbb A^\op$; and
  as a general fact, if $\D$ is the codescent object of an internal
  category $\mathbb D$ in $\cat{CAT}$, then $\D^\op$ is the codescent
  object of $\mathbb D^{\co \op}$.
%
  Putting these facts together, we conclude  that $\C[\mathbb
  A^\ast]^\op$ is the codescent object of the crossed internal category
  $(\mathbb A^\op)^{\co\op} = \mathbb A^\co$, whose description in
  Weber's framework agrees with that given by our construction above.
\end{Rk}

\section{$\mathsf P$-split epis}
\label{sec:mathsf-c-split}
Consider once again the isomorphisms~\eqref{eq:12} defining the
category of left weak maps of an \awfs. If the category $\D$ therein
has binary coproducts, then the double category $\dcat{SplEpi}(\D)$ is
the double category of algebras for an \awfs $\mathsf{SE}(\D)$ on $\D$; whence
by Lemma~\ref{lem:1} and Theorem~\ref{thm:characterise}, we may
rewrite~\eqref{eq:12} as an isomorphism
\begin{equation*}
\Cat(\Wkl, \D) \cong \Lax(\,(\mathsf L, \mathsf R), \mathsf{SE}(\D)\,)\rlap{ .}
\end{equation*}
It is natural to try and see these isomorphisms as the action on homs
of an adjunction between categories and \awfs. However, there is an
obstruction to doing so: the Kleisli categories in the image of the
apparent left adjoint $\cat{Wk}_\ell(\thg)$ will typically not admit
binary coproducts, while the apparent right adjoint $\mathsf{SE}(\thg)$
can only be defined for those categories which do so.

The reason for this difficulty is that we should not be constructing
an adjunction between \awfs and categories, but rather one between
\awfs and \emph{comonads}. The left adjoint will send an \awfs
$(\mathsf L, \mathsf R)$ to its cofibrant replacement comonad, while
the right adjoint will send a comonad $\mathsf P$ to the following
\emph{$\mathsf P$-split epi} \awfs.

\subsection{$\mathsf P$-split epis}
\label{sec:mathsf-p-split}
Given a comonad $\mathsf P$ on $\C$, a \emph{$\mathsf P$-split epi} is
a map $g \colon A \to B$ equipped with a section $p \colon B
\rightsquigarrow A$ of $g$ in the Kleisli category $\Kl{P}$: thus,
a map $p \colon PB \to A$ of $\C$ with $gp = \epsilon_B \colon PB \to
B$. The $\mathsf P$-split epis are the vertical arrows of a concrete
double category over $\C$, most efficiently described as a pullback
\begin{equation}
\cd{ \pushoutcorner \mathsf P\text-\dcat{SplEpi}(\C) \ar[r] \ar[d]_V &
  \dcat{SplEpi}(\Kl{P}) \ar[d]^U \\
  \Sq{\C} \ar[r]_-{\Sq{G}} & \Sq{\Kl{P}}}\label{eq:28}
\end{equation}
along the cofree functor $G \colon \C \to \Kl{P}$. As in
Example~\ref{sec:projective} we see that $V$ will exhibit $\mathsf
P\text-\dcat{SplEpi}(\C)$ as the double category of $\mathsf
R$-algebras for an \awfs on $\C$ so long as $V_1 \colon \mathsf
P\text-\cat{SplEpi}(\C) \to \C^\atwo$ has a left adjoint. From our
explicit description of the $\mathsf P$-split epis, it is easy to see
that this will be so whenever $\C$ admits binary coproducts, with the
unit of the free $\mathsf P$-split epi on $f$ being
given as on the left in
\begin{equation}
  \cd{A \ar[d]_{f} \ar[r]^-{\iota_{A}} & A+PB \ar@<-3pt>[d]_{\spn{f,\epsilon_B}} \\
B \ar[r]_-{1} & B \ar@<-3pt>@{~>}[u]_{\iota_{PB}}} \qquad \
\cd{
A \ar[d]_{\iota_A} \ar[rr]^1 &  & A \ar[d]^{\iota_A} \\
A + PB \ar[r]_{1 + \Delta_B} & A + PPB \ar[r]_{1 + P\iota_{PB}} & A +
P(A + PB)\rlap{ .}
}\label{eq:52}
\end{equation}
The comonad $\mathsf L$ for this \awfs thus sends $f$ to $\iota_{A}
\colon A \to A + PB$ and has counit $(1,\spn{f,\epsilon_B}) \colon
\iota_{A} \to f$, while its comultiplication can be calculated
according to Remark~\ref{rk:2} as having $f$-component as on the right
above. In particular, the $\mathsf P$-split epi \awfs has cofibrant
replacement comonad $\mathsf P$.

Though we shall not need the general $\mathsf L$-coalgebras here, a
straightforward calculation shows that, whenever $A \in \C$ and $b
\colon B \to PB$ is a $\mathsf P$-coalgebra, we obtain $\mathsf
L$-coalgebra structure on $\iota_A \colon A \to A+B$; and that when
$\C$ is a lextensive category and $P$ preserves pullbacks
along coproduct injections, every $\mathsf L$-coalgebra arises in this way.

\subsection{The \awfs--comonad adjunction}
\label{sec:awfs-comon-adjunct}
We are now ready to construct the adjunction between \awfs and
comonads alluded to above. We write $\cat{AWFS}^+_\mathrm{lax}$ for
the full sub-$2$-category of $\Lax$ on those \awfs whose underlying
categories admits finite coproducts, and $\cat{CMD}^+_\mathrm{lax}$
for the corresponding $2$-category of comonads.

\begin{Thm}
  \label{thm:9}
  There is a $2$-adjunction
\begin{equation}
\cd[@C+4.5em]{ \cat{CMD}^+_\mathrm{lax}
  \ar@<-4.5pt>[r]_-{(\thg)\text-\textrm{split epi}}
  \ar@{<-}@<4.5pt>[r]^{\substack{\text{cofibrant replacement}}}
    \ar@{}[r]|{\bot} & \cat{AWFS}^+_\mathrm{lax}} 
  \label{eq:10}
\end{equation}
with invertible counit.
\end{Thm}
\begin{proof}
  The assignation sending a comonad $(\C, \mathsf P)$ to the functor
  {$G^\mathsf P \colon \C \to \Kl{P}$} is the action on objects of a
  $2$-functor $\cat{CMD}_\mathrm{lax} \to \Cat^\atwo$. It follows that
  the assignation sending $(\C, \mathsf P)$ to the left-hand arrow
  of~\eqref{eq:28} is the action on objects of a $2$-functor
  $\cat{CMD}_\mathrm{lax} \to \Dbl^\atwo$. By the preceding argument
  and Theorem~\ref{thm:characterise}, the restriction of this
  $2$-functor to $\cat{CMD}^+_\mathrm{lax}$ factors through
  $\cat{AWFS}^+_\mathrm{lax}$; this defines the right adjoint
  $2$-functor.

  We now show that, for any $(\C, \mathsf L, \mathsf R) \in
  \cat{AWFS}^+_{\mathrm{lax}}$, the cofibrant replacement comonad
  $(\C, \cof)$ provides the value at $(\C, \mathsf L, \mathsf R)$ of a left adjoint to this
  $2$-functor. Indeed, to give a morphism $(\C, \cof) \to (\D,
  \mathsf P)$ in $\cat{CMD}^+_\mathrm{lax}$ is to give a
  square as on the left in
  \[
  \cd[@C-0.5em]{
    \C \ar[d]_{G^\cof} \ar[r]^F & \D \ar[d]^{G^\mathsf P} & 
    \DAlg{R} \ar[d] \ar[r]^-{} & \dcat{SplEpi}(\Kl{P}) \ar[d] & 
    \DAlg{R} \ar[d] \ar[r]^-{} & \mathsf P\text-\dcat{SplEpi}(\D) \ar[d] \\
    \Kl{\cof} \ar[r]_{\bar F} & \Kl{P} & 
    \Sq{\C} \ar[r]_-{\Sq{G^\mathsf P F}} & \Sq{\Kl{P}} & 
    \Sq{\C} \ar[r]_-{\Sq{F}} & \Sq{\D}\rlap{ .}
    }
\]
By Theorem~\ref{thm:WeakMaps}, this is equally to give a square as in
the centre; and since~\eqref{eq:28} is a pullback, this is equivalent
to giving a square as on the right. But this is equally to give a lax
\awfs morphism $(\C, \mathsf L, \mathsf R) \to (\D, \mathsf L_\mathsf
P, \mathsf R_\mathsf P)$ into the $\mathsf P$-split epi \awfs on $\C$,
as required. The naturality of these bijections in $(\D, \mathsf P)$
is straightforward; to obtain the bijection on $2$-cells, replace
$(\D, \mathsf P)$ by $(\D^\atwo, \mathsf P^\atwo)$ above. Finally, the
invertibility of the counit is the fact that the cofibrant replacement
comonad of the $\mathsf P$-split epi \awfs is $\mathsf P$.
\end{proof}

Clearly, the $2$-adjunction of the previous proposition is fibred over
$\Cat$; on restricting to the fibre over a fixed category $\C$ with
finite coproducts, we obtain the following corollary, telling us that
the category of comonads on $\C$ with finite coproducts embeds
coreflectively into the category of \awfs on $\C$. Note the reversal
of the adjoints, caused by the fact that $\Awfs{\C}$ and
$\cat{CMD}(\C)$ embed \emph{contravariantly} into
$\cat{AWFS}^+_{\mathrm{lax}}$ and $\cat{CMD}^+_{\mathrm{lax}}$.
\begin{Cor}
  \label{cor:3}
For any category $\C$ with finite coproducts, there is an adjunction
  \[
  \cd[@C+4.8em]{ \Awfs{\C} \ar@<-4.5pt>[r]_-{\text{cofibrant
        replacement}} \ar@{<-}@<4.5pt>[r]^{(\thg)\text-\textrm{split
        epi}} \ar@{}[r]|{\bot} & \cat{CMD}(\C)}
\]
with invertible unit.
\end{Cor}
This restricted form of the adjunction---in the dual monad theoretic form of Section~\ref{sec:mathsf-t-split}
 below---was first constructed in~\cite[Theorem~34]{Hebert2011Weak}.

\subsection{A universal property of the split epi AWFS} 
\label{sec:univ-prop-split}
The following result---a corollary of Theorem~\ref{thm:9}
above---can be seen as providing a justification as to \emph{why}
split epimorphisms and cofibrant objects are so closely connected. In
its statement, we write $\cat{AWFS}_{\mathrm{cocts}}$ for the
$2$-category obtained by restricting $\Oplax$ to \awfs on cocomplete,
locally small categories and to cocontinuous oplax morphisms between
them.

\begin{Cor}
  \label{cor:2}
  The \awfs for split epis on $\cat{Set}$ is the free cocomplete \awfs
  on a cofibrant object; by which we mean that it is a
  birepresentation for the $2$-functor
  \[
  \cat{Cof}(\thg) \colon \cat{AWFS}_{\mathrm{cocts}} \to \Cat\rlap{ ,}
  \]
  sending an \awfs $(\mathsf L, \mathsf R)$ to its category
  $\Coalg{Q}$ of algebraically cofibrant objects.
\end{Cor}
\begin{proof}
  We may restrict~\eqref{eq:10} to locally small,
  cocomplete categories and to functors with chosen left adjoint,
  yielding a $2$-adjunction
  $\cat{CMD}^\mathrm{coc}_\mathrm{radj} \leftrightarrows
  \cat{AWFS}^\mathrm{coc}_\mathrm{radj}$.
  The doctrinal adjunction of Section~\ref{sec:algebr-weak-fact},
  restricted to locally small, cocomplete categories, yields
  $(\cat{AWFS}^\mathrm{coc}_{\mathrm{radj}})^{\co\op} \cong
  \cat{AWFS}^\mathrm{coc}_{\mathrm{ladj}}$;
  similarly
  $(\cat{CMD}^\mathrm{coc}_{\mathrm{radj}})^{\co\op} \cong
  \cat{CMD}^\mathrm{coc}_{\mathrm{ladj}}$ and so 
  we obtain a $2$-adjunction
  $\cat{AWFS}^\mathrm{coc}_\mathrm{ladj} \leftrightarrows
  \cat{CMD}^\mathrm{coc}_\mathrm{ladj}$ with the left and right
  $2$-adjoints now \emph{reversed}. In particular, we have
  isomorphisms of categories
\[\cat{AWFS}^{\mathrm{coc}}_{\mathrm{ladj}}(\,\mathsf{SE}(\cat{Set}), (\C, \mathsf L, \mathsf R)) \cong \cat{CMD}^{\mathrm{coc}}_{\mathrm{ladj}}(\,(\cat{Set}, 1), (\C, \mathsf Q)\,)\]
$2$-natural in $(\C, \mathsf L, \mathsf R)$. Since a functor
$\cat{Set} \to \C$ into a locally small category admits a right
adjoint just when it preserves colimits, the above isomorphisms give
rise to equivalences
\[\cat{AWFS}_{\mathrm{cocts}}(\,\mathsf{SE}(\cat{Set}), (\C, \mathsf L, \mathsf R)) \simeq \cat{CMD}_{\mathrm{cocts}}(\,(\cat{Set}, 1), (\C, \mathsf Q)\,)\]
pseudonatural in $(\C, \mathsf L, \mathsf R)$. Now by~\cite[Theorems~7
and 12]{Street1972The-formal}, the category on the right is
$2$-naturally isomorphic to $\cat{CAT}_\mathrm{cocts}(\cat{Set},
\Coalg{Q})$---noting that $\Coalg{Q}$ is cocomplete, since $\C$ is so
and the forgetful functor creates colimits. Now as $\cat{Set}$ is the
free cocomplete category on the object $1$, the functor
$\cat{CAT}_\mathrm{cocts}(\cat{Set}, \Coalg{Q}) \to \Coalg{Q}$ given
by evaluation at $1$ is an equivalence; combining with the previous
equivalences yields the required pseudonatural equivalence
$\cat{AWFS}_\mathrm{cocts}(\,\mathsf{SE}(\cat{Set}), (\C, \mathsf L,
\mathsf R)\,) \simeq \Coalg{Q}$.
\end{proof}

\subsection{$\mathsf T$-split monos and sketches}
\label{sec:mathsf-t-split}
By dualising the preceding results, we obtain the \awfs for
\emph{$\mathsf T$-split monos} associated to any monad $\mathsf T$ on
a category $\C$ with binary products. Its coalgebras, the $\mathsf
T$-split monos, are maps $f \colon A \to B$ equipped with a Kleisli
retraction $B \rightsquigarrow A$; and if $\C$ also has a terminal
object, then its algebraically fibrant objects are the $\mathsf
T$-algebras. Dualising Theorem~\ref{thm:9}, we obtain an
adjunction as on the left in:
\begin{equation*}
  \cd[@C+3.3em]{ \cat{MND}^\times_\mathrm{oplax} \ar@<-4.5pt>[r]_-{(\thg)\text{\emph{-split mono}}} \ar@{<-}@<4.5pt>[r]^{(\C, \mathsf L, \mathsf R) \mapsto
      (\C, \fib)}
    \ar@{}[r]|{\bot} & \cat{AWFS}^\times_\mathrm{oplax}} 
  \qquad \quad
  \cd[@C+2em]{ \Awfs{\C} \ar@<-4.5pt>[r]_-{(\thg)\text{\emph{-split mono}}} \ar@{<-}@<4.5pt>[r]^{(\mathsf L, \mathsf R) \mapsto
      \fib}
    \ar@{}[r]|{\bot} & \cat{Mnd}(\C)}
\end{equation*}
with invertible counit; while restricting to a fixed category $\C$
with finite products, we obtain an adjunction as on the right, also
with invertible counit. 

We spell out this dual case primarily in order to highlight a connection
between the \awfs for $\mathsf T$-split monos and the \emph{sketches}
of~\cite{Kinoshita1999Sketches}. Given a finitary monad $\mathsf T$ on
a locally presentable category $\C$, \cite{Kinoshita1999Sketches}
defines a \emph{$\mathsf T$-sketch} to be given by:
\begin{enumerate}[(i)]
\item A small family of $4$-tuples $\D = (c_i, d_i, j_i \colon c_i \to d_i,
  k_i \colon d_i \to Tc_i)$ with each $c_i$ and $d_i$ finitely
  presentable and with $k_i j_i = \eta_{c_i}$ for each $i$;
\item An object $X$ of $\C$ and a $\D$-indexed family of maps $\phi_i
  \colon d_i \to X$;
\end{enumerate}
and a (strict) \emph{model} of this sketch in a $\mathsf T$-algebra
$(A,a)$ to be a map $f \colon X \to A$ rendering commutative each
square on the left in
\begin{equation}\label{eq:14}
\cd{
d_i \ar[r]^{k_i} \ar[d]_{\phi_i} & Tc_{i} \ar[d]^{a \cdot T(f\phi_i j_i)}
& & & c_i \ar[r]^{\psi_i} \ar[d]_{\alg j_i} & X\rlap{ .} \\
X \ar[r]_f & A & & & d_i \ar[ur]_{\phi_i}}
\end{equation}

Now, it is immediate that the family $\D$ of $4$-tuples in the
definition of sketch is nothing other than a family of $\mathsf
T$-split monos $(j_i, k_i) \colon c_i \to d_i$. So we can see a
$\mathsf T$-sketch as being given by an object $X$ and a family $\D$
of commuting triangles as to the right above wherein $j_i$ bears
$\mathsf T$-split mono structure.

What about models for sketches? Note that a special case of the
lifting property of any \awfs is that each map $c \to A$ into an
algebraically fibrant object admits a canonical extension along any
$\mathsf L$-map $c \to d$. Since $\mathsf T$ is the fibrant
replacement monad of the $\mathsf T$-split mono \awfs, this means in
particular that for any diagram as on the left in
\[
\cd{
c \ar[d]_{\alg j} \ar[r]^-h & \alg A\\
d \ar@{.>}[ur]|-{\bar h}
} \qquad \qquad \qquad
\cd{
c_i \ar[d]_{\alg j_i} \ar[r]^-{\psi_i} & X \ar[r]^f & \alg A\\
d_i \ar[ur]_-{\phi_i}}
\]
wherein $\alg j = (j,k)$ is a $\mathsf T$-split mono and $\alg A =
(A,a)$ is a $\mathsf T$-algebra, there is a canonical filler $\bar h$;
explicitly, we may calculate that $\bar h = a.Th.k$. Comparing with
the definition of model given above, we see that a model for a sketch
$(X, \D)$ in a $\mathsf T$-algebra $\alg A = (A,a)$ is a map $f \colon X \to A$
such that, for each triangle in $\D$ the composite triangle right
above is a canonical lifting triangle.

This suggests the following general definition. Given an \awfs
$(\mathsf L, \mathsf R)$ on a category $\C$ with terminal object, an
\emph{$(\mathsf L, \mathsf R)$-sketch} is given by an object $X \in
\C$ together with a family $\D$ of triangles as to the right
of~\eqref{eq:14} wherein $\alg j_i$ is equipped with $\mathsf L$-map
structure. A \emph{model} for a sketch $(X, \D)$ in an algebraically
fibrant object $\alg A$ is a morphism $f \colon X \to A$, composition
with which sends chosen triangles in $\D$ to canonical $(\mathsf L,
\mathsf R)$-lifting triangles. The sketches
of~\cite{Kinoshita1999Sketches} are then the specialisation of this
definition to the $\mathsf T$-split mono \awfs. A key result
in~\cite{Kinoshita1999Sketches} is one assuring the existence of
initial models for sketches, and it not hard to see that these
arguments may be carried out for sketches relative to any accessible
\awfs on a locally presentable category.

\section{Two dimensional monad theory}
\label{sec:two-dimens-monad}
We conclude this paper by examining two applications of the theory of
weak maps. The first is to the theory of \emph{$2$-monads}. If $\mathsf T$ is
a $2$-monad on a $2$-category $\C$, then in addition to the usual
Eilenberg--Moore $2$-category $\TAlgs$, one also has the
$2$-categories $\TAlgl$ and $\TAlgp$, whose objects are again the
$\mathsf T$-algebras, but whose maps are the \emph{lax} or
\emph{pseudo} algebra morphisms
$(f, \phi) \colon (A,a) \to (B,b)$. A lax morphism involves a $1$-cell
$f \colon A \to B$ and a $2$-cell
$\phi \colon b \cdot Tf \Rightarrow f \cdot a \colon TA \to B$
satisfying two coherence axioms~\cite[\S3.5]{Kelly1974Review}; a
pseudomorphism is the same but with $\phi$ invertible. Our objective
in this section, as presaged in Examples~\ref{ex:1} above, will be to
exhibit $\TAlgl$ and $\TAlgp$ as the categories of left weak maps for
suitable \awfs on $\TAlgs$. Note we are being slightly loose here,
since the theory of weak maps operates at the level of mere categories
while $\TAlgs$, $\TAlgp$ and $\TAlgl$ are in fact $2$-categories. To
capture the two-dimensional structure would require the theory of
enrichment over the \emph{monoidal \awfs} of~\cite{Riehl2013Monoidal};
as these are beyond our present scope, we will consider $\TAlgs$,
$\TAlgp$ and $\TAlgl$ as mere $1$-categories and proceed accordingly.

Before continuing, let us note that the result we are aiming for
allows us to reconstruct the main Theorem~3.13
of~\cite{Blackwell1989Two-dimensional}; this states that for an
accessible $2$-monad $\mathsf T$ on a complete and cocomplete
$2$-category $\C$, the inclusion functors
\begin{equation}
\cd{
I \colon \TAlgs \to \TAlgl \qquad \text{and} \qquad J \colon \TAlgs \to \TAlgp
}\label{eq:29}
\end{equation}
have left adjoints.\footnote{In~\cite{Blackwell1989Two-dimensional},
  $\TAlgl$, $\TAlgp$ and $\TAlgs$ are considered as $2$-categories,
  and $I$ and $J$ as $2$-functors, which are shown to have left
  $2$-adjoints; but as they explain, this two-dimensional aspect of
  the result is easily deduced from the one-dimensional one in the
  presence of cotensor products in $\C$.} In fact, this result is now
a triviality: the inclusions~\eqref{eq:29} may be identified with the
cofree functors into the categories of weak maps, so their left
adjoints are given simply by the corresponding forgetful functors.

We suppose for the remainder of this section that $\mathsf T$ is an
accessible monad on a complete and cocomplete $2$-category $\C$. Since
$\C$ is cocomplete, we may form the \awfs for lalis on the underlying
category $\C_0$. Note that its underlying monad $\mathsf R$, being
induced as in Example~\ref{ex:4} by left Kan extension and
restriction, is cocontinuous. Moreover, the underlying ordinary monad
$\mathsf T_0$ is accessible since $\mathsf T$ is so, whence by
\cite[Proposition~14]{Bourke2014AWFS1} the \awfs for lalis admits a
projective lifting along $U \colon \TAlgs \to \C_0$, which we term the
\awfs for \emph{$U$-lalis} on $\TAlgs$. By starting instead from the
\awfs for retract equivalences on $\C_0$ we obtain the lifted \awfs
for \emph{$U$-retract equivalences} on $\TAlgs$; the underlying weak
factorisation system of this latter \awfs was constructed
in~\cite[\S4.4]{Lack2007Homotopy-theoretic}.

\label{sec:lax-oplax-pseudo}
As $\C$ admits an initial object, so too does $\TAlgs$, and so we may
form the category of left weak maps of any \awfs thereon. We will
prove that, in the case of $U$-lalis or $U$-retract equivalences the
inclusion $\TAlgs \to \Wkl$ may be identified with the inclusion of
$\TAlgs$ into $\TAlgl$ or $\TAlgp$; then the latter inclusions will
have left adjoints since the former ones do.

Consider first the case of $U$-lalis. If $f \colon A \to B$ in
$\TAlgs$ bears $U$-lali structure---meaning that it comes equipped with a
right adjoint section $p \colon UB \to UA$ of $Uf$ in $\C$---then by
the doctrinal adjunction
of~\cite[Proposition~1.3]{Kelly1974Doctrinal}, $p$ bears a
\emph{unique} structure of lax $\mathsf T$-algebra morphism $B
\rightsquigarrow A$ making it into a right adjoint section of $f$ in
$\TAlgl$. Using the unicity of the lax structure, it is easy to see
that this assignation yields a pullback of double categories as in the
left square below
\begin{equation*}
\cd{U\text-\dcat{Lali}(\TAlgs) \ar[r] \ar[d] & \dcat{Lali}(\TAlgl) \ar[d] \ar[r] & \dcat{SplEpi}(\TAlgl) \ar[d]\\
\Sq{\TAlgs} \ar[r]_{\Sq{J}} & \Sq{\TAlgl} \ar[r]_{\Sq{1}} & \Sq{\TAlgl} \rlap{ .}}
\end{equation*}
By Theorem~\ref{thm:WeakMaps} the composite
$K \colon U\text-\dcat{Lali}(\TAlgs) \to \dcat{SplEpi}(\TAlgl)$
induces an extension of $J \colon \TAlgs \to \TAlgl$ to a functor
$\bar J \colon \Wkl \to \TAlgl$, defined as follows. For each
$A \in \TAlgs$, form the free $U$-lali $\fr !_A \colon QA \to A$ with
underlying strict map $\epsilon_A \colon QA \rightarrow A$; applying
$K$ yields a splitting $q_A \colon A \rightsquigarrow QA$ for
$\epsilon_A$ in $\TAlgl$; and now $\bar J \colon \Wkl \to \TAlgl$ is
the identity on objects, and on maps sends $f \colon QA \to B$ to
$f \cdot q_A \colon A \rightsquigarrow B$. To complete the proof, it
suffices to show that $\bar J$ is an isomorphism; of course, it is
bijective on objects, and so it suffices to exhibit an inverse for
each function
$(\thg) \cdot q_{A} \colon \TAlgs(QA,B) \to \TAlgl(A,B)$.

So consider a lax morphism $f \colon A \rightsquigarrow B$. Since $\C$ is complete, \cite[Theorem~3.2]{Lack2005Limits} assures us that the arrow $f$ admits an oplax limit in $\TAlgl$, as to the left in:
\begin{equation*}
\cd{& B/f \ltwocell{d}{\lambda} \ar[dl]_{u} \ar[dr]^{v} \\
A \ar@{~>}[rr]_{f} & & B} \qquad \qquad
\cd{& A \ar[dl]_{1} \ar@{~>}[dr]^{f} \\
A \ar@{~>}[rr]_{f} & & B\rlap{ ,}}
\end{equation*}
whose projections are strict morphisms that \emph{jointly detect  strictness}---in the sense that any lax morphism $c \colon C \rightsquigarrow B / f$ with $uc$ and $vc$ strict, is itself strict.

Applying the universal property of $B/f$ to the cone to the right
above, we induce a unique lax morphism
$r \colon A \rightsquigarrow B/f$ with $ur=1$, $\lambda r=1$ and
$vr=f$. In fact, using the two-dimensional aspect of the universal
property, we see that $u \dashv r$ is a lali in $\TAlgl$, so that
$\alg u = (u, Ur) \colon B / f \to A$ is a $U$-lali in $\TAlgs$. The
map $(!_{B/f}, 1_A) \colon !_A \to u$ in $\C^\atwo$ thus induces a map
of $U$-lalis as on the left in:
\begin{equation*}
\cd{QA \ar[r]^{k} \ar[d]_{\fr !_{A}} & B/f \ar[d]^{\alg u}\\
A \ar[r]^{1} & A }
\quad \quad \quad
\cd{QA \ar[r]^{k} & B/f  \ar[r]^{v} & B\rlap{ ,} \\
A \ar@{~>}[u]^{q_{A}} \ar[r]^{1} & A \ar@{~>}[u]_{r} \ar@{~>}[ur]_{f}}
\end{equation*}
applying $K$ to which yields a morphism of split epis in $\TAlgl$; in
particular the diagram above right commutes, and so $vk \colon QA \to
B$ is a strict map with $(vk) q_A = f$ as required. It remains to
show unicity of $vk$: so given $s \colon QA \to B$ with $s q_A =
f$, we must show that $s=vk$. Writing $\gamma_A$ for the unit of the
adjunction $\epsilon_A \dashv q_A$ in $\TAlgl$, we have $s\gamma_A
\colon s \Rightarrow sq_A \epsilon_A = f\epsilon_A$, so that we have a
cone in $\TAlgl$ as on the left in:
\begin{equation*}
\cd{& QA \ltwocell{d}{s \gamma_{A}} \ar[dl]_{\epsilon_{A}} \ar[dr]^{s} \\
A \ar@{~>}[rr]_{f} & & B} \qquad \qquad 
\cd{QA \ar[r]^{\overline{s}} \ar[d]_{\fr !_A} & B/f \ar[d]^{\alg u}\\
A \ar[r]^{1} & A\rlap{ .} }
\end{equation*}
As both projections are strict, there is a unique strict factorisation
$\overline{s} \colon QA \to B/f$ with $u \overline s = \epsilon_A$ and
$\lambda \overline s = s \gamma_A$ and $v \overline s = s$. Using the
universal property of $B/f$, it is easy to show that $(\overline s,1)$
is a morphism of lalis $\fr !_A \to \alg u$ as on the right above,
whence by freeness of $\fr !_A$, we have $\overline s = k$ and so $s =
v \overline s = v k$ as required. Thus we have shown:
\begin{Thm}
  Let $\mathsf T$ be an accessible $2$-monad on a complete and
  cocomplete $2$-category. The category $\TAlgl$ of algebras and lax
  morphisms is equally the category of left weak maps for the \awfs
  for $U$-lalis; in particular, the inclusion $\TAlgs \to \TAlgl$ has
  a left adjoint with counit a lali in $\TAlgl$.
\end{Thm}
Repeating this argument with $U$-retract equivalences in place of
$U$-lalis allows us to identify the left weak maps for the lifted
\awfs with the category $\TAlgp$ of $\mathsf T$-algebra
pseudomorphisms, and so to deduce the existence of a left adjoint to $\TAlgs
\to \TAlgp$. This pseudo case should be contrasted
with~\cite[Theorem~4.12]{Lack2007Homotopy-theoretic}: whereas our
result shows that the cofibrant replacement comonad of the $U$-retract
equivalence \awfs gives rise to the adjunction between strict and
pseudo algebra maps, Theorem 4.12 of \emph{ibid}.~starts by assuming
the adjunction between strict and pseudo maps, and deduces that the
induced comonad provides a notion of cofibrant replacement for the
$U$-retract equivalence \awfs.

\section{DG-enriched monad theory}
\label{sec:homological}Our second application of the theory of weak
maps will be to the description of \emph{homotopy-coherent} morphisms
between algebras for a dg-monad. First let us recall some basic
definitions. Fixing a commutative ring $R$,
we write $\cat{DG}$ for the category of unbounded chain complexes of
$R$-modules
\[
\cdots\xrightarrow{\partial} X_1 \xrightarrow{\partial} X_0 \xrightarrow{\partial} X_{-1}
\xrightarrow{\partial} \cdots
\]
This has a symmetric monoidal structure, whose unit $I$ satisfies $I_0
= R$ and $I_k = 0$ for $k \neq 0$, whose binary tensor is defined
by
\[
(X \otimes Y)_n = \textstyle\sum_{p+q=n} X_p \otimes Y_q \quad \text{and} \quad
\partial(x \otimes y) = \partial x \otimes y + (-1)^{\mathrm{deg}(x)}x
\otimes \partial y\rlap{ ,}
\]
and whose symmetry $\sigma \colon X \otimes Y \to Y \otimes X$
satisfies
$\sigma(x \otimes y) = (-1)^{\mathrm{deg}(x)\mathrm{deg}(y)}y \otimes
x$.
A \emph{dg-category}~\cite{Kelly1965Chain} is a category enriched in
$\cat{DG}$; it thus has $R$-modules of maps $\C(A,B)_n$ between any
two objects, whose elements we write as $f \colon A \to_n B$ and call
\emph{graded maps of degree $n$}. Graded maps have a bilinear
composition which adds degrees, and a differential $\partial$ such that
$\partial(gf) = \partial g \cdot f +
(-1)^{\mathrm{deg}(g)}g \cdot \partial f$
and $\partial(1_A) = 0$. Note that maps in the underlying ordinary
category of $\C$ are graded maps of degree $0$ with \emph{zero}
differential; we call such maps \emph{chain maps} and write them as
$f \colon A \to B$ with no subscript.

\subsection{Homotopy-coherent maps}
\label{sec:weak-maps-1}
For the rest of this section, we suppose that $\mathsf T$ is an
accessible dg-enriched monad on a cocomplete dg-category $\C$. We
have, of course, the dg-category $\TAlgs$ of $T$-algebras and
\emph{strict} maps; its object are $T$-algebras, and its graded maps
$f \colon (A,a) \to_i (B,b)$ are maps $f \colon A \to_i B$ in $\C$
such that $b \cdot Tf = f \cdot a$. However, we may also define a dg-category
$\TAlgw$ of \emph{homotopy-coherent maps}: its objects are again $T$-algebras,
while its graded maps $f \colon A \rightsquigarrow_i B$ are families
$(f_n \colon T^nA \to_{n+i} B)_{n \in \mathbb N}$ of graded maps in
$\C$ such that $f_n \cdot T^j \eta T^{n-j-1} = 0$ for all $0 \leqslant j
\leqslant n-1$. The differential of $f$ is the family
\[
(\partial f)_n = (\partial f_n) - (-1)^{i} [b \cdot Tf_{n\mn 1} +
f_{n\mn 1}\sum_{j=1}^{n-1}(-1)^{j} T^{j\mn 1} \mu_{T^{n\mn j\mn 1}A} +
(-1)^{n} f_{n\mn 1} \cdot T^{n\mn 1}a]\rlap{ .}
\]
The identity $A \rightsquigarrow_0 A$ has components $(1,0,0,\dots)$,
while the composite of $f \colon A \rightsquigarrow_i B$ with $g
\colon B \rightsquigarrow_k C$ is given by
\[
(gf)_n = \textstyle\sum_{p+q=n} (-1)^{pi} g_p \cdot T^p f_q\rlap{ .}
\]
There is an evident forgetful dg-functor $\TAlgw \to \C$
sending $f \colon
(A,a) \rightsquigarrow_i (B,b)$ to $f_0 \colon A \to_i B$, and an inclusion
$J \colon \TAlgs \to \TAlgw$ which is the identity on objects and
sends $f$ to $(f,0,0,0,\dots)$.

\begin{Ex}
  \label{ex:7}
  Let $\C$ and $\D$ be dg-categories with $\C$ small and $\D$
  cocomplete. We may identify the dg-functor category $[\C, \D]$ with
  the Eilenberg--Moore category $\TAlgs$ for the dg-monad $\mathsf T$
  on $[\ob \C, \D]$ induced by left Kan extension and restriction
  along the inclusion $\ob \C \to \C$. In this case, the dg-category
  $\TAlgw$ has dg-functors as objects and hom-objects given by the
  complexes of \emph{homotopy-coherent transformations} from $F$ to
  $G$ as defined in~\cite[\S3.1]{Tamarkin2007What}, for example.
\end{Ex}

\subsection{Homological lalis and $U$-lalis}
\label{sec:homological-lalis}
Our objective is now to show that, under suitable assumptions on
$\mathsf T$, we may obtain the underlying ordinary category of
$\TAlgw$ as the category of weak maps associated to an \awfs on
$\TAlgs$. The \awfs in question will be constructed using the
following analogue of the $2$-categorical lalis of
Example~\ref{ex:3}.

\begin{Defn}
  \label{def:1}
  A \emph{homological lali} in a dg-category $\A$ is a chain map $g
  \colon A \to B$ together with a section $q \colon B \to A$ and a
  graded map $\xi \colon A \to_1 A$ with $\partial \xi = 1 - qg$ and
  $g\xi = \xi q = \xi\xi = 0$. Homological lalis on $\C$ form a
  category $\cat{Lali}(\A)$, wherein a morphism $(u,v) \colon
  (g,q,\xi) \to (g',q',\xi')$ is a commuting square of chain maps
  $(u,v) \colon g \to g'$ such that $uq = q'v$ and $u\xi = \xi'u$.
  Homological lalis compose according to the formula
\[
A \xrightarrow{(g', q', \xi')} B \xrightarrow{(g,q,\xi)} C \qquad
\mapsto \qquad A
\xrightarrow{(gg', q'q, \xi' + q'\xi g')} C
\]
and so we obtain a double category $\dcat{Lali}(\A) \to
\dcat{Sq}(\A_0)$ of lalis which is concrete and right-connected over
$\A_0$ via the double functor sending $(g,q,\xi)$ to $g$.
\end{Defn}
Arguing as in Example~\ref{ex:4}, the homological lalis will comprise
the right class of an \awfs in any sufficiently cocomplete dg-category
$\A$; in this case, the colimits required are those for \emph{mapping
  cylinders}~\cite[\S1.5.5]{Weibel1994An-introduction}. The underlying
weak factorisation system of this \awfs is the (cofibration, trivial
fibration) part of a model structure on $\A$, constructed
in~\cite[Theorem~2.2]{Christensen2002Quillen}, and there called the
\emph{relative model structure} for the trivial projective class.

In particular, the cocomplete dg-category $\C$ we are considering
admits the \awfs for homological lalis; as the dg-monad $\mathsf T$
thereon is accessible, we may argue as in the preceding section to
projectively lift this \awfs to one on the underlying category of
$\TAlgs$, which as before we call the \awfs for \emph{$U$-lalis}. As
$\C$ admits an initial object, so too does $\TAlgs$, and so we may
form the category $\Wkl$ of left weak maps associated to the \awfs for
$U$-lalis. In the remainder of this section, we will prove the
following result; the side condition on preservation of codescent
objects will be explained shortly.

\begin{Thm}\label{thm:dg}
  Let $\mathsf T$ be an accessible dg-monad on a cocomplete
  dg-category. If $\mathsf T$ preserves codescent objects, then the
  underlying category of $\TAlgw$ is equally the category of left weak
  maps for the \awfs for $U$-lalis; in particular, the inclusion 
  $\TAlgs \to \TAlgw$ has a left adjoint with counit a lali in
  $\TAlgw$.
\end{Thm}
A key step in proving this is the following lemma, which is an
analogue of the doctrinal adjunction~\cite{Kelly1974Doctrinal} we used
in the preceding section.

\begin{Lemma}
  \label{lem:12}
  If $(g, f_0, \epsilon_0) \colon (B,b) \to (A,a)$ is a $U$-lali in
  $\TAlgs$, then there is a unique lali $(g, f, \epsilon)$ in $\TAlgw$
  with $Uf = f_0$ and $U\epsilon = \epsilon_0$ and $\epsilon_0 f_k =
  \epsilon_0 \epsilon_k = 0$ for all $k$.
\end{Lemma}
\begin{proof}
The zero components of $f$ and $\epsilon$ are $f_0$ and
$\epsilon_0$, and for $n>0$ we take
\[
f_{n} = \epsilon_0b \cdot Tf_{n-1} \qquad \text{and} \qquad
\epsilon_{n} = -\epsilon_0b \cdot Te_{n-1}\rlap{ .}
\]
A short calculation shows that $(g,f,\epsilon)$ is indeed a lali with
$Uf = f_0$ and $U\epsilon = \epsilon_0$ and $\epsilon_0 f_k =
\epsilon_0 \epsilon_k = 0$ for all $k$. Suppose now that $(f',
\epsilon')$ also satisfies these conditions; we show by induction on
$n$ that $f'_n = f_n$ and $\epsilon'_n = \epsilon_n$ for all $n$. The
case $n = 0$ is clear. For the inductive step, assume the result for
all $m < n$. It is easy to see that $\epsilon_p \cdot T^p f_q = 0$ for all
$p$ and $q$, and so
\[
0 = (\epsilon' f)_n = \textstyle\sum_{p+q = n} \epsilon_p' \cdot  T^p f_q = \epsilon_n'
 \cdot T^n f_0\rlap{ .}
\]
We now have that $-f'_n = -f'_n \cdot T^n g \cdot T^n f_0 = (1 - f'g)_n \cdot T^nf_0 =
(\mathbf d\epsilon')_n \cdot T^n f_0$; and so that
\begin{align*}
-f'_n = (\mathbf d\epsilon')_n \cdot T^n f_0 & = \mathbf d(\epsilon'_n)T^nf_{0} +
b \cdot T\epsilon_{n-1} \cdot T^nf_0 - \epsilon_{n-1} \textstyle\sum_{j=0}^{n-1}(-1)^j
d_j T^nf_0\\
& = - \epsilon_{n-1} \textstyle\sum_{j=0}^{n-2}(-1)^j
T^{n-1} f_0 \cdot d_j + (-1)^{n}\epsilon_{n-1} \cdot T^{n-1}b \cdot T^nf_0\\
& = (-1)^{n}\epsilon_{n-1}T^{n-1}(b \cdot Tf_0) = -f_n\rlap{ ;}
\end{align*}
the calculation that $\epsilon'_n = \epsilon_n$ is identical in form.
\end{proof}
It follows easily from the existence and uniqueness established in
this result that there is a lifting of the inclusion $J \colon \TAlgs
\to \TAlgw$ to a double functor as on the left in:
\begin{equation*}
\cd{U\text-\dcat{Lali}(\TAlgs) \ar[r] \ar[d] & \dcat{Lali}(\TAlgw) \ar[d] \ar[r] & \dcat{SplEpi}(\TAlgw) \ar[d]\\
\Sq{\TAlgs} \ar[r]_{\Sq{J}} & \Sq{\TAlgw} \ar[r]_{\Sq{1}} & \Sq{\TAlgw} \rlap{ .}}
\end{equation*}
So by Theorem~\ref{thm:WeakMaps}, the composite 
$U\text-\dcat{Lali}(T\text-\cat{Alg}_s) \to
\dcat{SplEpi}(T\text-\cat{Alg}_w)$ induces an extension of $J \colon
\TAlgs \to \TAlgw$ to a functor
\begin{equation}
\bar J \colon \Wkl \to \TAlgw\label{eq:7}\rlap{ .}
\end{equation}
To complete the proof of Theorem~\ref{thm:dg}, it remains to show that
this functor is an isomorphism. In the two-dimensional case, we did
this using the fact that oplax limits of morphisms lift along the
forgetful 2-functor $\TAlgl \to \C$. The analogous limit in the
dg-enriched case is the \emph{mapping path space}, but unfortunately,
it does not appear to be true that these limits lift along the
forgetful dg-functor $\TAlgw \to \C$. We are therefore forced to take
a different approach: we will compute the free $U$-lali
$\fr !_A \colon QA \to A$ explicitly, and use this to show directly
that~\eqref{eq:7} is an isomorphism. It turns out that this free
$U$-lali is obtained precisely by the familiar \emph{bar construction}
of~\cite{Cartan1956Homological}.

\subsection{Codescent objects and homological lalis}
\label{sec:weak-maps-codescent}
Let $\mathbf{\Delta}
\colon \Delta \to \cat{DG}$ be the functor sending $[n]$ to the
standard homological $n$-simplex
\[
\mathbf \Delta[n]_k = \bigoplus_{f \colon [k] \rightarrowtail [n]}
\!\!\! R \qquad \text{with} \qquad \partial(f) =
\textstyle\sum_{j=0}^n (-1)^j f \delta_i\rlap{ .}
\]
By the \emph{codescent object} of a simplicial diagram $X \colon
\Delta^\op \to \C$ in a dg-category, we mean the weighted colimit $\abs X =
\mathbf \Delta \star X$. The colimiting cocone comprises graded maps
$\iota_n \colon X_n \to_n \abs X$ such that $\iota_n s_j = 0$ for all $0
\leqslant j < n$ and such that
\[ \partial(\iota_0) = 0 \qquad \text{and} \qquad \partial(\iota_n) =
\iota_{n-1}\textstyle\sum_{j = 0}^n (-1)^j d_j \text{ for $n > 0$}\rlap{ ;}
\]
composition with these maps induces a bijection between maps $f \colon
\abs X \to_i B$ in $\C$, and families of maps $f_n \colon X_n \to_{n+i} B$
such that $f_ns_j = 0$ for all $0 \leqslant j < n$. 
If $X$ is an \emph{augmented} simplicial object---so equipped with a
map $d_0 \colon X_0 \to X_{-1}$ satisfying $d_0 d_0 = d_0 d_1$---then
there is an induced chain map $p \colon \abs X \to X_{-1}$
characterised by $p \iota_0 = d_0$ and $p \iota_j = 0$ for $j > 0$. 

There are essentially well-known conditions under which this induced
$q$ will be a homological lali. By a \emph{contraction}~\cite[\S
III.5]{Goerss1999Simplicial} on an augmented simplicial object $X$, we
mean maps $s_{-1}$ as in
\[
\cd[@C+2em]{
{} \ar@{}[r]|{\cdots} & X_2 \ar@<14pt>[r]|{d_0} \ar@<-14pt>[r]|{d_2} \ar[r]|{d_1}
\ar@<7pt>@{<-}[r]|{s_0} \ar@<-7pt>@{<-}[r]|{s_1}
\ar@<14pt>@/^1em/@{<--}[r]^{s_{-1}} & X_1
\ar@<7pt>[r]|{d_0} \ar@<-7pt>[r]|{d_1}
\ar@<7pt>@/^1em/@{<--}[r]^{s_{-1}} \ar@{<-}[r]|{s_0} & X_0
\ar@{->}[r]|{d_0} \ar@/^1em/@{<--}[r]^{s_{-1}} & X_{-1}
}
\]
satisfying $d_0 s_{-1} = 1$ and $d_{i+1}s_{-1} = s_{-1} d_i$ and
$s_{j+1}s_{-1} = s_{-1} s_j$. Such a contraction induces a structure
of homological lali on the comparison $q \colon \abs X \to X_{-1}$; indeed,
  $p$ has the section $q = \iota_0 s_{-1} \colon X_{-1} \to
  X_0 \to \abs X$, and we define $\xi \colon \abs X
  \to_1 \abs X$ to
  be the unique graded map with
  \[
  \xi \iota_n = \iota_{n+1} s_{-1} \colon X_n \to X_{n+1}
  \to_{n+1} \abs X\rlap{ .}
  \]
  Straightforward calculation now shows that $(p,q,\xi)$ is a
  homological lali.

\subsection{Codescent objects and weak maps}
\label{sec:codesc-objects-weak}
Using the preceding result, we may now give an explicit construction
of the free $U$-lali $\fr !_A \colon QA \to A$ for a dg-monad $\mathsf
T$ which preserves codescent objects. Given $a \colon TA \to A$ a
$T$-algebra, we consider in $\TAlgs$ its \emph{bar complex}, the
augmented simplicial object $A_{\bullet}$ as in the solid part of
\[
\cd[@C+2em]{
\cdots\ & T^3 A \ar@<15pt>[r]|{\mu T} \ar@<-15pt>[r]|{T^2a}
\ar[r]|{T\mu} \ar@<15pt>@/^1em/@{<--}[r]^{\eta T^2}
\ar@<7.5pt>@{<-}[r]|{T\eta T} \ar@<-7.5pt>@{<-}[r]|{T^2\eta} & T^2 A
\ar@<7.5pt>[r]|{\mu} \ar@<-7.5pt>[r]|{Ta} \ar@{<-}[r]|{T\eta} 
\ar@<7.5pt>@/^1em/@{<--}[r]^{\eta T}& TA \ar[r]|a
\ar@/^1em/@{<--}[r]^{\eta}& A
}
\]
where each vertex except for the rightmost one bears its free algebra
structure. We will continue to use $s_j$ and $d_j$ to denote the face and
degeneracy maps; explicitly we have that:
  \begin{equation*}
s_j  \colon T^n A
\to T^{n+1}A = T^{j+1} \eta_{T_{n-j-1}A} \quad \text{for $-1 \leqslant j < n$;}
\end{equation*}
\begin{equation*}
  \text{and} \qquad d_j \colon T^{n+1}A \to T^n = \begin{cases}T^j \mu_{T^{n-j-1}A}
 &\text{for $0 \leqslant j < n$;}\\
  T^n a & \text{for $j = n$.}
\end{cases}
\end{equation*}
Note that we have $Ts_i = s_{i+1}$ and $Td_i = d_{i+1}$. Let $QA$ be
the codescent object of this bar complex and $p \colon QA \to A$
the comparison map to the augmentation. Note that the algebra
structure $\bar a \colon TQA \to QA$ is uniquely determined by the
equations
\[
\bar a \cdot T\iota_n = \iota_{n+1}d_0 \colon T^{n+2} A \to QA\rlap{ .}
\]
Since $T$ preserves codescent
objects so does $U \colon \TAlgs \to \C$, and so $UQA$ is the
codescent object of $UA_{\bullet}$ in $\C$. The unit maps $\eta$ equip
$UA_\bullet$ with a contraction, and so $Uq$ is part of a lali
$(Up, q, \xi)$ in $\C$; thus $(p,q,\xi) \colon QA \to A$ is a
$U$-lali. 
\begin{Prop}
  \label{prop:2}
  With the above assumptions and notation, $(p, q, \xi)$ is the
  free $U$-lali $\fr !_A \colon QA \to A$.
\end{Prop}
\begin{proof}
  Let $(g, f, \epsilon) \colon B \to A$ be a $U$-lali. We must show
  there is a unique $\mathsf T$-algebra map $h \colon QA \to B$
  comprising a map of $U$-lalis $(p, q, \xi) \to (g, f, \epsilon)$;
  this means that $gh = p$ and $f = hq$ and $\epsilon h = h\xi$. Now
  if we are to have $f = hq$ then
  $f = hq = h\iota_0 s_{-1} = h\iota_0\eta_A\colon A \to B$; since
  $h\iota_0$ is to be a map of $T$-algebras, this forces
  $h\iota_0 = b \cdot Tf \colon TA \to B$. Similarly, if we are to
  have $\epsilon h = h\xi$, then
  $\epsilon h \iota_n = h \xi \iota_n = h \iota_{n+1}s_{-1} = h
  \iota_{n+1}\eta_{T^{n+1}A} \colon T^{n+1} A \to B$;
  which since $h \iota_{n+1}$ is a $T$-algebra map forces
  $h\iota_{n+1} = b \cdot T(\epsilon h \iota_n) \colon T^{n+2} A \to
  B$.
  Since $h \colon QA \to B$ is determined by its precomposites with
  the $\iota_n$'s, this proves the uniqueness of $h$, and it remains
  only to check that these definitions do indeed yield a map of
  $U$-lalis. So let the algebra maps $h_n \colon T^{n+1} A \to_n B$ be
  defined by $h_0 = b \cdot Tf$ and
  $h_{n+1} = b \cdot T(\epsilon h_n)$; it follows easily that we have
\begin{gather*}
h_0 s_{-1} = f \quad\text{and} \quad h_{n+1} s_{-1} = \epsilon h_n \quad\text{and}\quad
h_n d_0 = b \cdot Th_n\rlap{ .}
\end{gather*}
We first prove $h_n s_j = 0$
for all $0 \leqslant j < n$. We have that
\[
h_{n} s_j = b \cdot T(\epsilon h_{n-1}) \cdot T(s_{j-1}) = b \cdot T(\epsilon
h_{n-1} s_{j-1})\rlap{ .}
\]
If $j = 0$, then this is zero since $\epsilon h_0 s_{-1} = \epsilon f
= 0$ and $\epsilon h_{n+1} s_{-1} = \epsilon \epsilon h_n = 0$. If $j
> 0$, this is zero by induction on $j$. So by the universal property
of $QA$, there is a unique algebra map $h \colon QA \to_0 B$ with $h
\iota_n = h_n$. It remains to check that:
\begin{itemize}
\item $f = hq$ and $\epsilon h = h \xi$; which is forced by the method
  of definition.
\item $gh = p$; which follows since we have $gh_0 = gb \cdot Tf = a \cdot Tg \cdot Tf = a =
  p\iota_0$, and for $n>0$ that
$gh_n = gb \cdot T(\epsilon h_{n-1}) = a \cdot T(g\epsilon h_{n-1})
= 0 = p\iota_n$.
\item $h$ is a chain map; which is to say that $\partial(h_n) =
  h\partial(\iota_n)$ for each $n$. But $
  \partial(h_0) = \partial(b \cdot Tf) = 0 = h\partial(\iota_0)$; and
  \begin{align*}
\partial(h_1) & = \partial(b \cdot T(\epsilon h_0)) = b \cdot T((1-fg)h_{0} - \epsilon \partial
(h_{0}))= b \cdot T(h_0 - fa)\\
&= b \cdot Th_{0} - b \cdot Tf \cdot Ta= h_{0}d_0 - h_{0}d_1 = h\partial(\iota_1)\rlap{ ;}
  \end{align*}
and for
  $n\geqslant 1$ we show that $\partial(h_n) = h\partial(\iota_n)$
  by induction and the calculation
\begin{align*}
  \partial(h_n) & = \partial(b \cdot T(\epsilon h_{n-1})) =
  b \cdot T((1-fg)h_{n-1} - \epsilon \partial (h_{n-1}))\\ &= b \cdot T(h_{n-1} -
  \epsilon h\partial(\iota_{n-1}))
  = b \cdot Th_{n-1} - b \cdot T(\epsilon h_{n-2} \textstyle\sum_{j=0}^{n-1} (-1)^j d_{j})\\
  &=h_{n-1}d_0 - h_{n-1}\textstyle\sum_{j=0}^{n-1} (-1)^j d_{j+1} =
  h \partial(\iota_{n})
  \rlap{ .} \qedhere
\end{align*}
\end{itemize}
\end{proof}
Given this result, we are now in a position to complete the proof of
Theorem~\ref{thm:dg} by showing that the functor~\eqref{eq:7} is an
isomorphism. Of course, it is bijective on objects, and on morphisms
is defined as follows. For each $A \in \TAlgs$, form the free $U$-lali
$(p, q, \xi) \colon QA \to A$ as in the preceding result; now apply
Lemma~\ref{lem:12} to obtain a splitting
$\bar q \colon A \rightsquigarrow QA$ for $p$ in $\TAlgw$, which by direct
calculation has components
  \[
  \bar q_n = \iota_n s_{-1} = \iota_n \eta_{T^n A}\colon T^n A \to
  T^{n+1} \to_n QA\rlap{ .}
  \]
  Now the action of~\eqref{eq:7} on morphisms sends $f \colon QA \to
  B$ to $f \cdot \bar q \colon A \rightsquigarrow B$. But it is easy to see
  that this assignation is invertible; given a weak map $g \colon A
  \rightsquigarrow B$, the unique strict $T$-algebra map $\bar f
  \colon QA \to B$ inducing it is determined by the conditions
\[
\bar f \iota_n = b \cdot Tf_n \colon T^{n+1} A \to B\rlap{ .}
\]
This completes the proof of Theorem~\ref{thm:dg}.

\begin{Ex}
  \label{ex:8}
  Let $A$ be a unital dg-algebra---a monoid in $\cat{DG}$, and let
  $\mathsf T$ be the monad $A \otimes (\thg)$ on $\cat{DG}$. In this
  case, $QA \to A$ is the classical bar resolution of $A$~\cite[X,~\S
  6]{Cartan1956Homological}. More generally, if $\A$ is a small
  dg-category, and $\mathsf T$ is the dg-monad on $[\ob \A, \cat{DG}]$
  whose algebras are dg-modules $\A \to \cat{DG}$, then $QA \to A$ is
  the bar resolution $Y \circ A$ described
  in~\cite[Example~6.6]{Keller1994Deriving}.
\end{Ex}
It is worth also pointing out some of the examples which
Theorem~\ref{thm:dg} does not encompass. The category of dg-algebras
is itself monadic over $\cat{DG}$, and as well as the usual strict
morphisms of dg-algebras there are also the well-known weak
(=$A_\infty$) morphisms. It is natural to attempt to re-find these by
lifting the \awfs for homological lalis from $\cat{DG}$ to the
category of dg-algebras. While this is certainly possible, it is does
not fall under the scope of Theorem~\ref{thm:dg}, since the monad for
dg-algebras on $\cat{DG}$ is \emph{not} a dg-monad. It is, however, a
monad induced by a dg-operad, and in a sequel to this paper we will
examine homotopy-coherent maps of algebras over dg-operads from the
perspective of \awfs; we will see that they arise from a
\emph{dendroidal}~\cite{Moerdijk2007Dendroidal} bar resolution which
can also be understood in terms of the bar--cobar construction for
operad algebras~\cite[Chapter~11]{Loday2012Algebraic}.








\end{document}